\pdfoutput=1
\documentclass[11pt]{amsart}
\usepackage{graphicx,dsfont}
\usepackage{amssymb}
\usepackage{mathtools}
\usepackage{tikz-cd}
\usepackage{color}
\usepackage[title]{appendix}
\usepackage{amsmath, mathrsfs, stmaryrd, color,slashed,bbm}
\usepackage{lipsum}
\usepackage{float}
\usepackage{verbatim}
\usepackage{comment}
\usepackage[justification=centering]{caption}
\usepackage[utf8]{inputenc}
\usepackage[english]{babel}
\usepackage{enumitem}
\usepackage{booktabs}

\newtheorem{theorem}{Theorem}

\newtheorem{definition}[theorem]{Definition}
\newtheorem{proposition}[theorem]{Proposition}

\newtheorem{lemma}[theorem]{Lemma}

\theoremstyle{definition}
\newtheorem{example}[theorem]{Example}
\newtheorem{remark}[theorem]{Remark}

\numberwithin{theorem}{section}
\numberwithin{equation}{section}

\allowdisplaybreaks

\begin{document}

\title[Toroidal Positive Energy Theorem and Rigidity Results]
{The Positive Energy Theorem for Asymptotically Hyperboloidal Initial Data Sets With Toroidal Infinity and Related Rigidity Results}

\author[Alaee]{Aghil Alaee}
\address{\parbox{\linewidth}{Aghil Alaee\\
Department of Mathematics, Clark University, Worcester, MA 01610, USA\\
Center of Mathematical Sciences and Applications, Harvard University,
Cambridge, MA 02138, USA}}
\email{aalaeekhangha@clarku.edu, aghil.alaee@cmsa.fas.harvard.edu}

\author[Hung]{Pei-Ken Hung}
\address{\parbox{\linewidth}{Pei-Ken Hung\\
School of Mathematics, University of Minnesota,
Minneapolis, MN 55455}}
\email{pkhung@umn.edu}

\author[Khuri]{Marcus Khuri}
\address{\parbox{\linewidth}{Marcus Khuri\\
Department of Mathematics, Stony Brook University,
Stony Brook, NY 11794, USA}}
\email{khuri@math.sunysb.edu}

\thanks{A. Alaee acknowledges the support of an AMS-Simons travel grant. M. Khuri acknowledges the support of NSF Grants DMS-1708798, DMS-2104229, and Simons Foundation Fellowship 681443.}


\begin{abstract}
We establish the positive energy theorem and a Penrose-type inequality for 3-dimensional asymptotically hyperboloidal initial data sets with toroidal infinity, weakly trapped boundary, and satisfying the dominant energy condition. In the umbilic case, a rigidity statement is proven showing that the total energy vanishes precisely when the initial data manifold is isometric to a portion of the canonical slice of the associated Kottler spacetime. Furthermore, we provide a new proof of the recent rigidity theorems of Eichmair-Galloway-Mendes \cite{EGM} in dimension 3, with weakened hypotheses in certain cases. These results are obtained through an analysis of the level sets of spacetime harmonic functions.
\end{abstract}

\maketitle

\section{Introduction}
\label{sec1}
\setcounter{equation}{0}
\setcounter{section}{1}

The positive energy theorem for asymptotically hyperbolic initial data sets with spherical infinity is well-studied. There is a vast literature on the subject, and we direct the interested reader to some of the most recent results \cite{BHKKZ1,BHKKZ,ChruscielDelay,HJM,Sakovich},
as well as the references therein. By contrast, much less is known about the nature of total energy for asymptotically (locally) hyperbolic data having a conformal infinity of positive genus. In fact, the question of positive energy appears to be significantly more delicate in this setting, since examples such as the Horowitz-Myers soliton \cite{HM} show that the positive energy theorem fails under the traditional hypotheses of completeness together with a satisfactory energy density condition. Here we will focus attention on asymptotically hyperboloidal initial data with toroidal infinity, and prove the positive energy theorem when a nonempty weakly trapped boundary is present. In addition, related rigidity results for compact initial data sets with boundary are also established.

Let $(M,g,k)$ be a 3-dimensional smooth initial data set for the Einstein equations, where $M$ is an orientable connected manifold with nonempty boundary, $g$ is a Riemannian metric, and $k$ is a symmetric 2-tensor representing the second fundamental form of an embedding into spacetime. The tensors $g$ and $k$ must satisfy the constraint equations
\begin{equation}
2\mu = R_g+(\mathrm{Tr}_{g}k)^{2}-|k|_{g}^{2},\qquad\quad
 J = \operatorname{div}_{g}\left(k-(\mathrm{Tr}_{g}k \right)g),
\end{equation}
where $R_g$ denotes scalar curvature and $\mu$, $J$ are the energy and momentum density of matter fields. We will say that the \textit{dominant energy condition} is satisfied if $\mu\geq |J|_g$. Note that if $k=\pm g$, then this condition implies the scalar curvature lower bound $R_{g}\geq -6$. Let $\Sigma$ denote a closed 2-sided surface in $M$ with null expansions $\theta_{\pm}=H\pm \mathrm{Tr}_{\Sigma}k$, where $H$ denotes the mean curvature of $\Sigma$ with respect to the unit normal that points towards a designated asymptotic end.
When the surface $\Sigma$ is viewed as embedded within spacetime, the null expansions represent the mean curvature in null directions, and hence measure the rate of change of area of shells of light emanating from the surface in the outward (toward infinity) future/past direction. Moreover, the null expansions arise as traces of the null second fundamental forms $\chi^{\pm}=II\pm k|_{\Sigma}$, where $II$ is the Riemannian second fundamental form of $\Sigma\subset M$.
A strong gravitational field is associated with an outer or inner trapped surface, that is, a surface for which $\theta_{+}<0$ or $\theta_{-}<0$. Moreover, $\Sigma$ will be referred to as weakly outer/inner trapped if $\theta_{+}\leq 0$ or $\theta_{-}\leq 0$, and will be referred to as a marginally outer/inner trapped surface (MOTS or MITS) if $\theta_+ =0$ or $\theta_- =0$.

An initial data set will be called \textit{asymptotically hyperboloidal with toroidal infinity}, if there is a compact set $\mathcal{K}\subset M$ such that its complement is diffeomorphic to a cylinder with torus cross-sections, and in the coordinates given by the diffeomorphism $\psi:(1,\infty)\times T^2 \rightarrow M\setminus\mathcal{K}$ the metric and extrinsic curvature satisfy
\begin{equation}\label{asymptotics}
\psi^* g=r^{-2} dr^2+r^2 \hat{g}+r^{-1}\mathbf{m} +Q_g, \qquad\quad
\psi^* \left(k+g\right)=r^{-1}\mathbf{p}+Q_k,
\end{equation}
where $r\in (1,\infty)$ is the radial coordinate, $\hat{g}$ is a flat metric and $\textbf{m}$, $\textbf{p}$ are symmetric two-tensors all on $T^2$, and $Q_g$, $Q_k$ are symmetric 2-tensors on $(1,\infty)\times T^2$ with the property that
\begin{equation}\label{alk3nhapgaj}
|Q_g|_{b}+r | \pmb{\nabla} Q_g|_b+r^2|\pmb{\nabla}^2 Q_g|_b=o(r^{-3}),\qquad\quad
|Q_k|_{b}+r |\pmb{\nabla} Q_k|_b=o(r^{-3}).
\end{equation}
Here $b$ is the model hyperbolic metric $r^{-2}dr^2+r^2\hat{g}$ on $(1,\infty)\times T^2$, and $\pmb{\nabla}$ is the Levi-Civita connection of $b$. Note that $b$ arises as a quotient of hyperbolic space $\mathbb{H}^3$ with identifications along horospheres, and is the induced metric on a constant time slice of the (toroidal) Kottler spacetime \cite{ChruscielSimon} with zero mass and cosmological constant $\Lambda=-3$. The quantity $\mathrm{Tr}_{\hat{g}}\left(3\mathbf{m}-2\mathbf{p}\right)$ on $T^2$ is referred to as the \textit{mass aspect function} and yields a well-defined total energy (\cite{CJL}, \cite{Michel}) if $r(\mu+|J|_g)\in L^1(M\setminus\mathcal{K})$, which is given by
\begin{equation}
E=\frac{1}{|T^2|}\int_{T^2}\mathrm{Tr}_{\hat{g}}\left(3\mathbf{m}-2\mathbf{p}\right) dA_{\hat{g}}
\end{equation}
where $|T^2|$ denotes $\hat{g}$-area.

In order to state the positive energy theorem, some restrictions on the topology of $M$ will be needed. In particular, we will make use of the so called \textit{homotopy condition} from \cite{EGM}, which generalizes the situation in which $M$ is a retraction onto a given 2-dimensional submanifold $\Sigma$. Namely, the manifold $M$ will be said to satisfy the homotopy condition with respect $\Sigma$, if there exists a continuous map $\rho:M\rightarrow\Sigma$ such that its composition with the inclusion map $\rho\circ i :\Sigma\rightarrow\Sigma$ is homotopic to the identity. Furthermore, we will say that $M$ satisfies the homotopy condition with respect to conformal infinity if the condition is satisfied for a coordinate torus in the asymptotic end.

\begin{theorem}\label{thm:PMT}
Let $(M,g,k)$ be a smooth orientable 3-dimensional asymptotically hyperboloidal initial data set with toroidal infinity satisfying the dominant energy condition. Suppose that the boundary is nonempty $\partial M\neq \varnothing$, that $M$ satisfies the homotopy condition with respect to conformal infinity, and $H_2(M,\partial M;\mathbb{Z})=0$. If the boundary is weakly outer trapped $\theta_{+}(\partial M)\leq 0$, then $E\geq 0$. Moreover, the same conclusion continues to hold if the boundary contains additional components which are weakly inner trapped and of genus zero.
\end{theorem}

It should be noted that the boundary need not have a specified topology except for the weakly inner trapped components. However, the hypothesis of a nonempty boundary cannot be removed if the conclusion is to remain valid. A counterexample to the boundaryless case is provided by the Horowitz-Myers geon with $k=-g$. The geon is a time slice of the Horowitz-Myers soliton, which gives a complete asymptotically locally hyperbolic Riemannian metric on the solid torus $D^2 \times S^1$ with constant scalar curvature $R_g =-6$ and negative mass. It is conjectured \cite{HM,Ewoolgar} that a complete Riemannian 3-manifold $(M,g)$ which is asymptotic to a Horowitz-Myers geon, and satisfies $R_g \geq -6$, must have total energy at least as large as that of the geon; furthermore, equality should hold between the energies only if the geometries are isometric.

Previous studies concerning lower bounds for the energy of asymptotically hyperboloidal initial data with toroidal infinity have focused on the umbilic case $k=- g$, with $R_g \geq -6$.
In particular, Chru\'{s}ciel-Galloway-Nguyen-Paetz \cite{CG,CGNP} have proven a version of the positive energy theorem minus the rigidity statement, assuming that there is a connected weakly outer trapped boundary and, in dimension 3, that the mass aspect function has a sign.
If the boundary is an outermost minimal surface, with at least one component having $T^2$-topology, then Lee-Neves \cite[Corollary 1.2]{LeeNeves} show that the mass aspect function has positive supremum. Furthermore, Barzegar-Chru\'{s}ciel-H\"{o}rzinger-Maliborski-Nguyen establish versions of the Horowitz-Myers conjecture under the assumption of axisymmetry, and also find supporting evidence in the perturbation regime, while Liang-Zhang \cite{Liang} prove a generalization. The case of equality $E=0$ has been treated by Huang-Jang \cite[Theorem 6]{HuangJang}, assuming that the positive energy inequality holds. Therefore, combining Theorem \ref{thm:PMT} with \cite[Theorem 6]{HuangJang} yields one method to establish the last (umbilic) statement of the following result. We will, however, provide an alternative approach based on a foliation by level sets of spacetime harmonic functions, which will in addition provide strong rigidity requirements in the general non-umbilic case.

\begin{theorem}\label{thm:rigidity}
If the energy vanishes $E=0$ under the assumptions of Theorem \ref{thm:PMT}, including the
trapped surface conditions on the boundary, then the following holds.
\begin{enumerate}
\item The manifold $M$ is diffeomorphic to $[1,\infty)\times T^2$.

\item Each level set $\Sigma_t=\{t\}\times T^2$ of the radial coordinate $t\in [1,\infty)$ is a MOTS, and in fact has vanishing null second fundamental form $\chi^+=0$.

\item The induced geometry on $\Sigma_t$ is that of a flat torus for all $t\in [1,\infty)$.

\item If $\nu_t$ denotes the unit normal to $\Sigma_t$ pointing towards infinity, then $\mu=|J|_g=-J(\nu_t)$ on $M$.
\end{enumerate}
Moreover, if in addition $k=- g$ then $(M,g)$ is isometric to the Kottler time slice $([1,\infty)\times T^2, b)$.
\end{theorem}

These two theorems are established using the level set technique associated with spacetime harmonic functions. This approach has recently been used to prove the positive mass theorem in the asymptotically flat and asymptotically hyperboloidal (spherical infinity) settings \cite{BHKKZ,BKKS,HKK}, and was inspired by the work of Stern \cite{Stern} where the level sets of harmonic maps were used to study scalar curvature on compact 3-manifolds. We refer the reader to the survey \cite{BHKKZ1} for these and other developments concerning the level set method. A function $u\in C^2(M)$ will be referred to as a \textit{spacetime harmonic function} if it satisfies the equation
\begin{equation}\label{shf}
\Delta u+\left(\mathrm{Tr}_g k\right)|\nabla u|=0,
\end{equation}
in which the left-hand side arises as the trace along $M$ of the \textit{spacetime Hessian}
\begin{equation}
\bar{\nabla}_{ij}^2 u:=\nabla_{ij}^2 u+k_{ij}|\nabla u|.
\end{equation}

Under the homotopy condition of Theorem~\ref{thm:PMT}, there exists a connected component of $\partial M$, denoted by $\partial_1 M$, such that $\partial_1 M$ cannot be separated from infinity by an embedded 2-sphere. See Section~\ref{sec2} for more details. We will say that a spacetime harmonic function $u$ is \textit{admissible} if it realizes constant Dirichlet boundary data together with $\partial_{\upsilon}u\geq 0$ on each boundary component, and there is at least one point on each boundary component except $\partial_1 M$ where $|\nabla u|=0$; here $\upsilon$ denotes the unit normal to $\partial M$ pointing towards infinity. The existence of admissible spacetime harmonic functions that asymptote to the radial coordinate function in the asymptotic end is shown in Sections \ref{sec4} and \ref{sec6} below. The following energy lower bound implies Theorem \ref{thm:PMT}, and is instrumental in the proof of Theorem \ref{thm:rigidity}, however, it holds without the assumption of an energy condition but adds the integrability condition for energy/momentum density that is associated with a well-defined total energy.

\begin{theorem}\label{energylb}
Let $(M,g,k)$ be a smooth orientable 3-dimensional asymptotically hyperboloidal initial data set with toroidal infinity, such that $r(\mu+|J|_g)\in L^1(M\setminus\mathcal{K})$. Suppose that the boundary is nonempty and weakly trapped, having at least one weakly outer trapped ($\theta_+ \leq 0$) component and with each weakly inner trapped ($\theta_- \leq 0$) component of genus zero. Assume further that $M$ satisfies the homotopy condition with respect to conformal infinity, and $H_2(M,\partial M;\mathbb{Z})=0$. Then there exists an admissible spacetime harmonic function $u$ that asymptotes to the radial coordinate in the asymptotic end, and induces the energy lower bound
\begin{equation}\label{thm1.1}
E\geq \frac{1}{|T^2|}\int_{M}\left(\frac{|\bar{\nabla}^2 u|^2}{|\nabla u|}+2\left(\mu+J(\nu)\right)|\nabla u|\right)dV
\end{equation}
where $\nu=\nabla u/|\nabla u|$. Moreover if in addition $k=-g$, the dominant energy condition holds, and the boundary is minimal $H=0$ instead of weakly trapped, then a Penrose-type inequality holds
\begin{equation}\label{thm1.2}
E\geq \mathcal{C}\frac{|\partial_1 M|}{|T^2|},
\end{equation}
where $\mathcal{C}=4\min_{\partial_1 M}\partial_{\upsilon}u>0$.
\end{theorem}

The methods used to prove this theorem may also be applied in the setting of compact manifolds with boundary. There we recover, with alternative arguments, a version of the main results obtained by Eichmair-Galloway-Mendes in \cite[Theorems 1.2 and 1.3]{EGM}, for dimension 3. The statement of our result differs from that of \cite[Theorem 1.2]{EGM}, in that the more restrictive hypothesis of vanishing second homology is included, while we allow for the more general situation of multiple untrapped boundary components. In contrast with \cite[Theorem 1.3]{EGM}, our assumption on $k$ leads to rigidity in the form of a warped product metric as opposed to a constant curvature model. Note that the boundary normal orientation is reversed and $k$ should be replaced by $-k$, when comparing with \cite{EGM}.

\begin{theorem}\label{thm:EGM} 
Let $(\Omega,g,k)$ be a smooth orientable 3-dimensional compact initial data set with boundary $\partial\Omega$, satisfying the dominant energy condition and $H_2(\Omega,\tilde{\partial} \Omega;\mathbb{Z})=0$. Suppose that $\tilde{\partial}\Omega:=\partial\Omega \setminus \partial_1^+\Omega$, and that the boundary may be decomposed into a disjoint union
\begin{equation}
\partial\Omega=\big(\sqcup_{i=1}^m\partial^{+}_{i}\Omega\big)\sqcup \big(\sqcup_{i=1}^\ell\partial^-_{i}\Omega\big),
\end{equation}
where the connected components are organized so that $\theta_{+}\left(\partial_i^+ \Omega\right)\geq 0$ with respect to the outer normal, and $\theta_{+}\left(\partial_i^- \Omega\right)\leq 0$ with respect to the inner normal. Moreover, assume that $\partial^+_{1}\Omega$ has positive genus, that $\partial^+_{i}\Omega$ is of zero genus for $i=2,\ldots,m$, and that $\Omega$ satisfies the homotopy condition with respect to $\partial^+_{1}\Omega$. Then the following statements hold.
\begin{enumerate}
\item There are only two boundary components, namely $m=\ell=1$. Indeed, $\Omega$ is diffeomorphic to $[0,t_0]\times T^2$ for some $t_0 >0$.
		
\item Each level set $\Sigma_t=\{t\}\times T^2$ of the radial coordinate $t\in [0,t_0]$ is a MOTS with respect to the normal $\nu_t$ pointing towards $\partial_1^+ \Omega$. In fact, these surfaces have vanishing null second fundamental form $\chi^+=0$.

\item The induced geometry on $\Sigma_t$ is that of a flat torus for all $t\in [0,t_0]$.

\item The energy and momentum densities satisfy $\mu=|J|_g=-J(\nu_t)$ on $\Omega$.
\end{enumerate}
Furthermore, if in addition $k=- \lambda g$ where $\lambda \in C^{\infty}(\Omega)$, then $(\Omega,g)$ is isometric to the warped product $\left([0,t_0]\times T^2, dt^2 + f(t)^2 \hat{g}\right)$, for some flat metric $\hat{g}$ on $T^2$ and a smooth positive radial function $f$ satisfying $\frac{d}{dt}\log f=\lambda$. In particular, if $\lambda=\lambda_0$ is a constant then $(\Omega,g)$ is of constant curvature $-\lambda_0^2$ and $\mu=|J|_g\equiv0$.
\end{theorem}

As pointed out in \cite{EGM}, the setting of Theorem \ref{thm:EGM} naturally arises in the context of Lohkamp's approach to the asymptotically flat version of the positive mass theorem. Namely, his method relies upon showing that an initial data set $(M,g,k)$ which is isometric to Euclidean space outside a bounded open set $U$ and with $k=0$ there as well, cannot have a strict  dominant energy condition $\mu>|J|_g$ on $U$ \cite[Theorem 2]{Lohkamp}. By taking a large cube enclosing $U$ and identifying two opposing pairs of sides, we obtain a compact initial data set $(\Omega,g,k)$ in which $\Omega$ is diffeomorphic to the connected sum $\left([0,1]\times T^2\right)\,\sharp\, N$ for some compact manifold $N$. The two boundary tori have vanishing null second fundamental form, and so they are MOTS. Therefore, if $H_2(\Omega,\tilde{\partial} \Omega;\mathbb{Z})=0$ then Theorem \ref{thm:EGM} confirms that this configuration with strict dominant energy condition on $U$ is not possible.

Motivated by the case of equality from the positive mass theorem in the asymptotically flat and asymptotically hyperboloidal (spherical infinity) settings, it is reasonable to suspect that
the rigidity statements of Theorems \ref{thm:rigidity} and \ref{thm:EGM}, when $k=-\lambda g$, might be generalized to produce an embedding of the initial data into a model flat spacetime.
In this direction, under a related hypothesis on $k$, Eichmair-Galloway-Mendes \cite[Theorem 6.1]{EGM} confirm such a result by showing that the data embed into a quotient of Minkowski space. In contrast, we provide in Section \ref{sec7} an example which shows that the restrictions on the structure of $k$ cannot be relaxed too far.

\begin{example}
There exist initial data $(M,g,k)$ satisfying the hypotheses of Theorem \ref{thm:rigidity} or \ref{thm:EGM} minus the assumption on the structure of $k$, while additionally exhibiting a vanishing mass aspect function (in the noncompact case) and vanishing energy and momentum densities $\mu=|J|_g=0$,
with the following properties. Unlike the conclusion of Theorems \ref{thm:rigidity} and \ref{thm:EGM}, the metric $g$ does not have a warped product structure, and in a departure from the conclusion of \cite[Theorem 6.1]{EGM} the initial data arise from a vacuum (with zero cosmological constant) pp-wave spacetime which is not flat.
\end{example}

This paper is organized as follows. In Section \ref{sec2} the topology of initial data sets is examined under the hypotheses of the main theorems, while in Section \ref{sec3} an integral identity for spacetime harmonic functions is presented. Existence and uniqueness of appropriate boundary value problems for spacetime harmonic functions is established in Section \ref{sec4}. The proof of Theorem \ref{thm:EGM} is presented in Section \ref{sec5}, while the proof of Theorems \ref{thm:PMT}, \ref{thm:rigidity}, and \ref{energylb} are presented in Section \ref{sec6}. Lastly, the example described in the preceding paragraph is given in Section \ref{sec7}.

\section{The Topology of Initial Data Sets}
\label{sec2}
\setcounter{equation}{0}
\setcounter{section}{2}

The purpose of this section is to record an attribute of the initial data that will be instrumental in controlling the level set topology for admissible spacetime harmonic functions. Recall that the statements of the main results described in the previous section imply that there is an embedded surface of positive genus with respect to which the initial data satisfies the homotopy condition. The desired property of the initial data, to be elucidated here, essentially says that this surface cannot be shielded from all other boundary components by a 2-sphere. To state this in a precise manner, we require certain definitions. Let $\Omega$ be a smooth 3-manifold and $\Sigma\subset \Omega$ be a properly embedded surface, that is, $\Sigma\cap\partial\Omega=\partial\Sigma$ and if this intersection is nonempty it is transverse. The notation $\Omega|\Sigma$ will be used to denote the \textit{splitting} of $\Omega$ along $\Sigma$. Intuitively, this is the possibly disconnected 3-manifold obtained from $\Omega$ by cutting along $\Sigma$. See \cite[page 3]{H} for a detailed description. Furthermore, if $n_1\leq n_2$ are two integers, then $\llbracket n_1,n_2\rrbracket$ will be used to denote the set of integers lying between and including $n_1$ and $n_2$.

\begin{definition}\label{def:separable}
Let $S_1$ and $S_2$ be connected components of $\partial\Omega$. We say $S_1$ and $S_2$ are separable by a 2-sphere, if there exists a properly embedded 2-sphere $\Sigma$ such that $S_1$ and $S_2$ belong to different connected components of $\Omega|\Sigma$. In this case, we say that $\Sigma$ separates $S_1$ and $S_2$.
\end{definition}

\begin{proposition}\label{lem:separable}
Let $\Omega$ be a compact, oriented, connected smooth 3-manifold with boundary $\partial\Omega$. Suppose that $\partial\Omega$ has at least two connected components, and may be decomposed as
\begin{equation}
\partial\Omega=\big(\sqcup_{i=1}^m\partial^{+}_{i}\Omega\big)\sqcup \big(\sqcup_{i=1}^\ell\partial^-_{i}\Omega\big)
\end{equation}
such that the following statements hold.
\begin{enumerate}
\item $\partial^+_1\Omega$ has positive genus.
\item $\partial^+_i\Omega$ are homeomorphic to 2-spheres for $i\in\llbracket 2,m \rrbracket$.
\item $\Omega$ satisfies the homotopy condition with respect to $\partial^+_1\Omega$.
\end{enumerate}
Then there exists $i_0\in \llbracket 1,\ell\rrbracket $ such that $\partial^-_{i_0}\Omega$ is not separable from $\partial^+_1\Omega$ by a 2-sphere.
\end{proposition}

\begin{remark}
There are large classes of manifolds beyond the model $[0,1]\times T^2$ which satisfy the conditions of Proposition \ref{lem:separable}. For instance, consider $\Omega=([0,1]\times S_g)\,\sharp\, N$, where $S_g$ is an orientable closed surface of genus $g\geq 1$ and $N$ is any orientable closed 3-manifold. Here $m=\ell=1$, $\partial_1^+\Omega=\{1\}\times S_g$, and $\partial_1^- \Omega=\{0\}\times S_g$. As discussed in \cite{EGM}, these manifolds satisfy the homotopy condition with respect to $\partial_1^+\Omega$. Furthermore, a calculation also shows that they satisfy the homology condition $H_2(\Omega,\tilde{\partial}\Omega;\mathbb{Z})=0$ of Theorem \ref{thm:EGM}, if $H_2(N;\mathbb{Z})=0$.
\end{remark}

The goal of this section is to establish Proposition \ref{lem:separable}.
We begin by recalling a formulation of the prime decomposition.
As above let $\Omega$ be a compact, oriented, connected 3-manifold with possibly non-empty boundary. $\Omega$ is called \textit{prime} if $\Omega=\Omega'\,\sharp\, \Omega''$ implies that either $\Omega'$ or $\Omega''$ is $S^3$; here $\sharp$ stands for the connected sum.
Moreover $\Omega$ is called \textit{irreducible} if every $2$-sphere $S^2 \subset \Omega$ bounds a 3-ball. It is well known \cite[Proposition 1.4]{H} that the only orientable prime 3-manifold which is not irreducible is $S^1\times S^2$. A version of the prime decomposition theorem \cite[Theorem 1.5]{H}
states that there exist irreducible 3-manifolds $\Omega_1,\Omega_2,\dots, \Omega_k$ and a nonnegative integer $l$ such that $\Omega$ is homeomorphic to the connected sum
\begin{equation}\label{equ:primedecomposition}
\Omega= \Omega_1\,\sharp\, \Omega_2\,\sharp\, \dots \,\sharp\,  \Omega_k\,\sharp\, l (S^1\times S^2).
\end{equation}
Furthermore, the decomposition is unique up to order and insertion or deletion of 3-spheres. We remark that the $\Omega_i$'s may have non-empty boundary.
In order to keep track of which prime summand contains particular boundary components of $\Omega$, it is helpful to utilize the concept of a \textit{reduction system}, which was introduced by Milnor \cite{M} (see also \cite[page 7]{H}) in the context of proving the uniqueness for the prime decomposition.

\begin{definition}\label{def:Sigma}
Let $\Omega$ be a compact, oriented, connected smooth 3-manifold, and let
$\Sigma$ be a family of disjoint, properly embedded 2-spheres in $\Omega$. We say that $\Sigma$ is a reduction system if
\begin{equation}\label{i2nbghh}
\Omega|\Sigma=\left(\sqcup_{j=1}^k Q_j\right) \sqcup \left( \sqcup_{j=1}^l R_j \right),
\end{equation}
where $Q_j$, $j\in\llbracket 1,k \rrbracket$ is obtained from the prime factor $\Omega_j$ by removing finitely many open 3-balls, while $R_j$, $j\in\llbracket 1,l \rrbracket$ is homeomorphic to $S^3$ with finitely many open 3-balls removed. We call $Q_j$ a punctured $\Omega_j$, and $R_j$ a punctured 3-sphere. Furthermore, by decomposing the boundary into components
$\partial\Omega=\sqcup_{i=1}^n \partial_i\Omega$, we may construct the reduction system correspondence $j_{\Sigma}:\llbracket 1,n \rrbracket\to \llbracket 1,k \rrbracket$ that associates to each component $\partial_i \Omega$ the unique punctured prime factor $Q_{j_{\Sigma}(i)}$ in which it is contained.
\end{definition}

Stated informally, $\Sigma$ is a reduction system if by cutting $\Omega$ along $\Sigma$, the prime decomposition is obtained where each punctured 3-sphere is associated with a $S^1\times S^2$ summand. Moreover, the reduction system correspondence map $j_{\Sigma}$ records the irreducible piece in which the boundary components $\partial_i\Omega$ lie. The next result plays an important role in the proof of Proposition \ref{lem:separable}, and states that a reduction system may be modified to avoid a given 2-sphere.

\begin{lemma}\label{lem:prime}
Let $\Sigma$ be a reduction system as in Definition \ref{def:Sigma}, and let $S\subset \Omega$ be a properly embedded 2-sphere. Then there exists another reduction system $\tilde{\Sigma}$ which is disjoint from $S$. Moreover, the boundary components of $\partial\Omega$ still belong to the same irreducible pieces, that is, $j_\Sigma(i)=j_{\tilde{\Sigma}}(i)$ for all $i\in \llbracket 1,n \rrbracket$.
\end{lemma}

\begin{proof}
We follow closely the arguments of \cite[page 7]{H}. The main idea is to gradually decrease the number of curves in $S\cap\Sigma$ by replacing $\Sigma$ with an `update' in a systematic manner. We begin by describing the types of updates that will be employed, and observe how the $Q_i$ and $R_i$ change in the process.

\begin{figure}
\centering
\begin{minipage}{0.45\textwidth}
\centering
\includegraphics[width=0.77\textwidth]{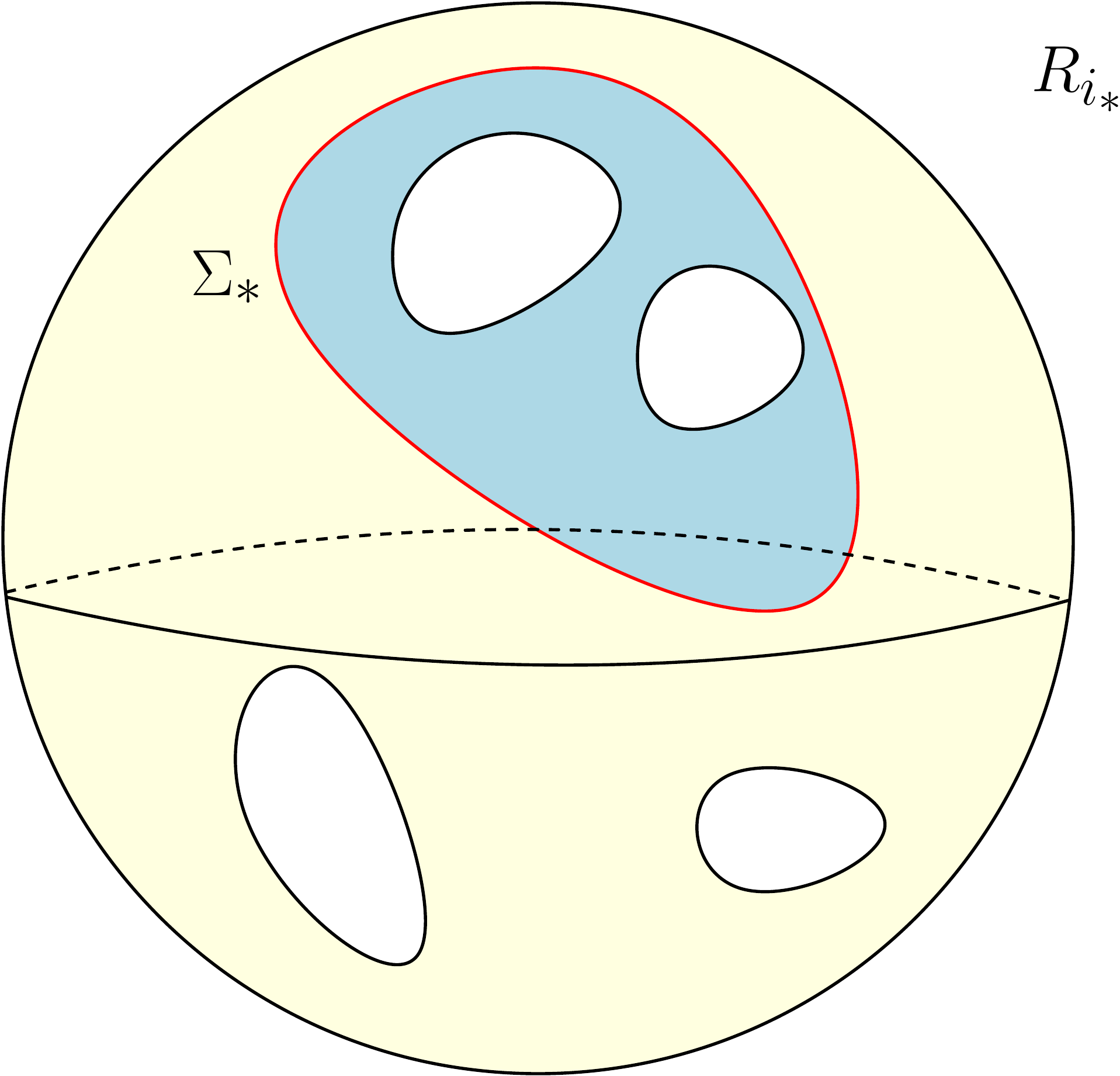} 
\end{minipage}\hfill
\begin{minipage}{0.45\textwidth}
\centering
\includegraphics[width=0.77\textwidth]{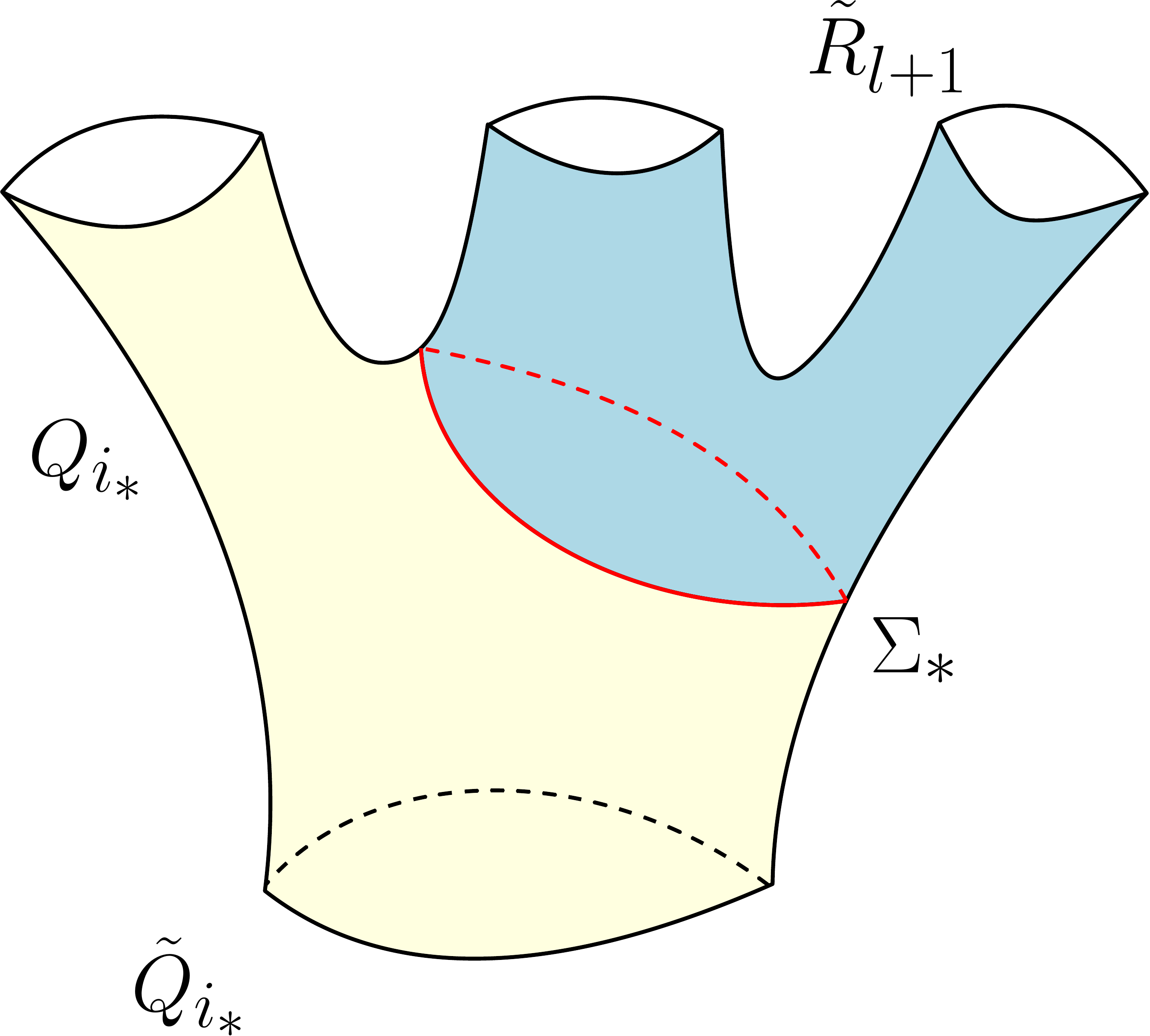} 
\end{minipage}
\caption{An illustration of the procedure that adds one more sphere into the reduction system.}\label{fig:1}
\end{figure}

The first type of modification consists of adding another properly embedded 2-sphere $\Sigma_{*}$, which is disjoint from $\Sigma$. More precisely, we set $\tilde{\Sigma}\coloneqq\Sigma\sqcup \Sigma_*$. Because $\Sigma_*$ is disjoint from $\Sigma$, $\Sigma_*$ must be properly embedded in either a $ Q_i$ or
$R_i$. Suppose that $\Sigma_*\subset R_{i_*} $, then all the $Q_i$ are unchanged and the $R_i$ for $i\neq i_*$ are also unchanged. Furthermore, $R_{i_*}$ splits into two punctured 3-spheres. See the left portion of Figure~\ref{fig:1} for an illustration. Now suppose that $\Sigma_*\subset   Q_{i_*}$, then the $Q_i$ with $i\neq i_*$ are unchanged and all $R_i$ are  unchanged as well. Moreover, $Q_{i_*}$ splits into two pieces $\tilde{Q}_{i_*}$ and $\tilde{R}_{l+1}$, where $\tilde{Q}_{i_*}$ is a punctured $\Omega_{i_*}$ which replaces $Q_{i_*}$ for the reduction system $\tilde{\Sigma}$, whereas $\tilde{R}_{l+1}$ is an additional punctured 3-sphere. Here we used the fact that $Q_{i_*}$ is a punctured $\Omega_{i_*}$, and that $\Omega_{i_*}$ is irreducible. See the right portion of Figure~\ref{fig:1} for an illustration.

The second type of update consists of eliminating a component $\Sigma_0 \subset\Sigma$ which satisfies the property that it lies at the transition between two different pieces of the decomposition \eqref{i2nbghh}, one of which is a punctured 3-sphere. More precisely, we require one of the following conditions to hold; where an overline bar denotes the closure of a set.
\begin{enumerate}
\item There exist $j_0\in \llbracket 1,k \rrbracket$ and $j_1\in \llbracket 1,l \rrbracket$ such that $\Sigma_0\subset  \overline{Q}_{j_0} \cap \overline{R}_{j_1} $.
\item There exist $j_0, j_1\in \llbracket 1,l \rrbracket$ with $j_0 \neq j_1$ such that $\Sigma_0\subset  \overline{R}_{j_0} \cap  \overline{R}_{j_1} $.
\end{enumerate}
In this situation we set $\tilde{\Sigma}\coloneqq \Sigma\setminus\Sigma_0$. Suppose that (2) holds. Then all the $Q_i$ are unchanged and the $R_i$ with $i\neq j_0,j_1$ are also unchanged. Moreover, the $R_{j_0}$ and $R_{j_1}$ may be glued together to form a single punctured $S^3$. See the left portion of Figure~\ref{fig:2} for an illustration. Thus, in this case, the decomposition \eqref{i2nbghh} associated with the new reduction system has one less punctured 3-sphere. Suppose now that (1) holds. Then all the $Q_i$ with $i\neq j_0$ remain unchanged, and the $R_i$ with $i\neq j_1$ are also unchanged. Furthermore, $Q_{j_0}$ and $R_{j_1}$ may be glued together to form $\tilde{Q}_{j_0}$, which is still a punctured $\Omega_{j_0}$ and replaces $Q_{j_0}$ in the decomposition for $\tilde{\Sigma}$. See the right portion of Figure~\ref{fig:2} for an illustration. We note that throughout these two update procedures, the maps $j_{\Sigma}, j_{\tilde{\Sigma}} :\llbracket 1,n \rrbracket\to \llbracket 1,k \rrbracket$ are identical.

\begin{figure}
\centering
\begin{minipage}{0.48\textwidth}
\centering
\includegraphics[width=\textwidth]{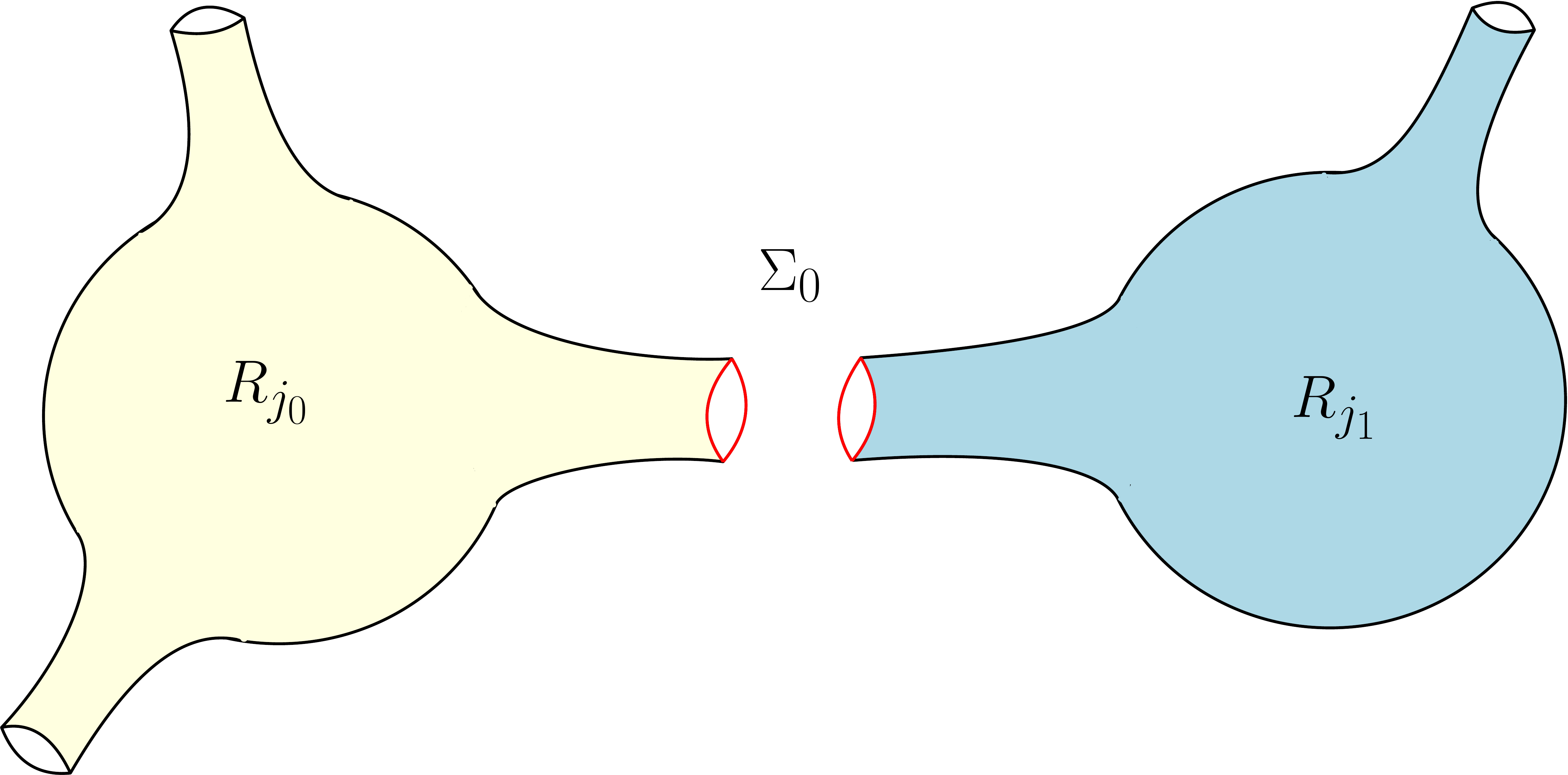} 
\end{minipage}\hfill
\begin{minipage}{0.48\textwidth}
\centering
\includegraphics[width=\textwidth]{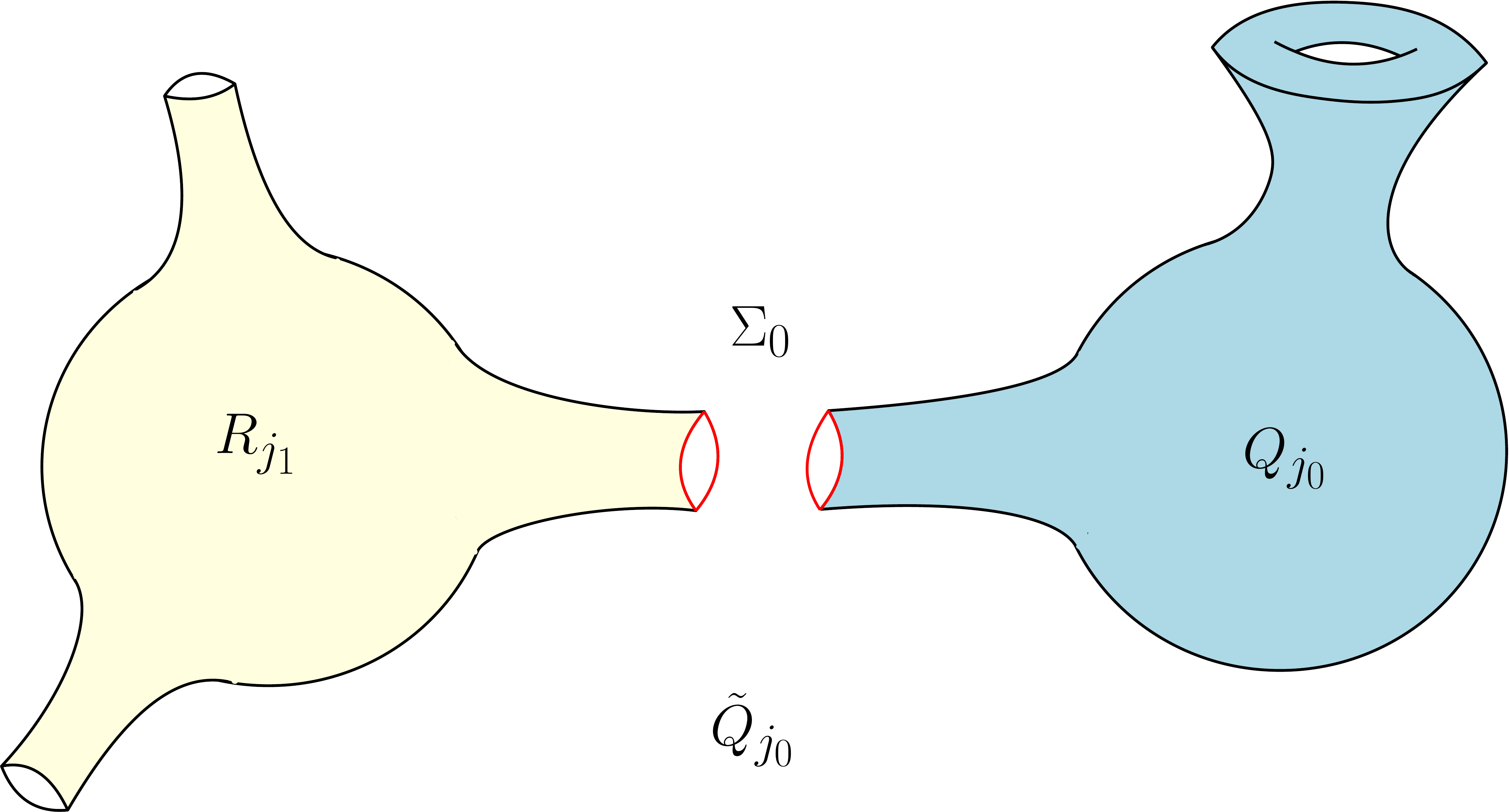} 
\end{minipage}
\caption{An illustration of the procedure that removes one sphere from the reduction system.}\label{fig:2}
\end{figure}

We are ready to begin the process of decreasing the number of curves in the intersection between $\Sigma$ and $S$. See Figure~\ref{fig:3} for an illustration. By perturbing $\Sigma$, we may assume that $\Sigma$ and $S$ intersect transversally. If $S\cap\Sigma=\varnothing$ then we are done, so assume that $S\cap\Sigma\neq \varnothing$. Let $C\subset S\cap \Sigma$ be an innermost closed curve which lies in some component $\Sigma_1$ of $\Sigma$. That is, $C$ bounds a closed disk $D$ in $S$ such that $\textup{int}(D)\cap \Sigma=\varnothing$. Let $D_\Sigma'$ and $D_\Sigma''$ be the closed disks on $\Sigma_1$ which are bounded by $C$. Define ${\Sigma}'_1\coloneqq D\cup D_\Sigma' $ and ${\Sigma}''_1\coloneqq D\cup D_\Sigma''$. Both ${\Sigma}'_1$ and ${\Sigma}''_1$ are properly embedded 2-spheres. By perturbing $\Sigma'_1$ and $\Sigma''_1$ appropriately, the result of which will be denoted by $\tilde{\Sigma'_1}$ and $\tilde{\Sigma}''_1$, we may achieve the following properties.
\begin{enumerate}
\item $\tilde{\Sigma}'_1$ and $\tilde{\Sigma}''_1$ are disjoint, and both are disjoint from $\Sigma$.
\item $\tilde{\Sigma}'_1$ and $\tilde{\Sigma}''_1$ intersect $S$ transversally. Furthermore, the number of closed curves in $S\cap(\tilde{\Sigma}'_1\sqcup \tilde{\Sigma}''_1)$, compared to $S\cap\Sigma_1$, is decreased by 1.
\item $\Sigma_1$, $\tilde{\Sigma}'_1$ and $\tilde{\Sigma}''_1$ bound a $3$-punctured $S^3$.
\end{enumerate}
In the depiction on the right-hand side of Figure~\ref{fig:3}, the above requirements can be satisfied by shrinking $\Sigma'_1$ and $\Sigma''_1$ to obtain $\tilde{\Sigma}'_1$ and $\tilde{\Sigma}''_1$.  We then replace $\Sigma$ by $\tilde{\Sigma}:=\big(\Sigma\sqcup \tilde{\Sigma}'_1\sqcup \tilde{\Sigma}''_1\big)\setminus \Sigma_1$. From (1), (3), and discussion concerning the two previous updates above, $\tilde{\Sigma}$ is still a reduction system and $j_\Sigma=j_{\tilde{\Sigma}}$. From (2), the number of closed curves in $S\cap \tilde{\Sigma}$, compared with $S\cap {\Sigma}$, is decreased by 1. This construction may be repeated until the new reduction system is disjoint from $S$, yielding the desired result.
\end{proof}

We will now utilize the ability to find a reduction system disjoint from a given 2-sphere, to show that if two boundary components of $\Omega$ are indeed separable by a 2-sphere, then they must belong to different irreducible pieces of the prime decomposition.

\begin{lemma}\label{cor:B7}
Let $\Omega$ be as in Definition \ref{def:Sigma}. Suppose that boundary components $\partial_{i_1}\Omega$ and $\partial_{i_2}\Omega$ are separable by a 2-sphere. Then for any reduction system $\Sigma$, the boundary components $\partial_{i_1}\Omega$ and $\partial_{i_2}\Omega$ must belong to different punctured irreducible pieces in the decomposition \eqref{i2nbghh}. That is, $j_{\Sigma}(i_1)\neq j_{\Sigma}(i_2)$.
\end{lemma}

\begin{proof}
Let $S$ be a properly embedded 2-sphere which separates $\partial_{i_1}\Omega$ and $\partial_{i_2}\Omega$. Proceeding by contradiction, suppose that there exists a reduction system $\Sigma$ with $j_{\Sigma}(i_1)= j_{\Sigma}(i_2)=:j$. By Lemma~\ref{lem:prime}, we may assume that $S\cap \Sigma=\varnothing$. Furthermore, since $S$ separates $\partial_{i_1}\Omega$ and $\partial_{i_2}\Omega$ in $\Omega$, we have that $S\cap Q_j$ must separate $\partial_{i_1}\Omega$ and $\partial_{i_2}\Omega$ in $Q_j$. On the other hand, $S\cap Q_j$ is a properly embedded 2-sphere within $Q_j$, as $S\cap\Sigma=\varnothing$. This, however, contradicts the fact that $\Omega_j$ is irreducible.
\end{proof}

\begin{figure}
\centering
\begin{minipage}{0.5\textwidth}
\centering
\includegraphics[width=\textwidth]{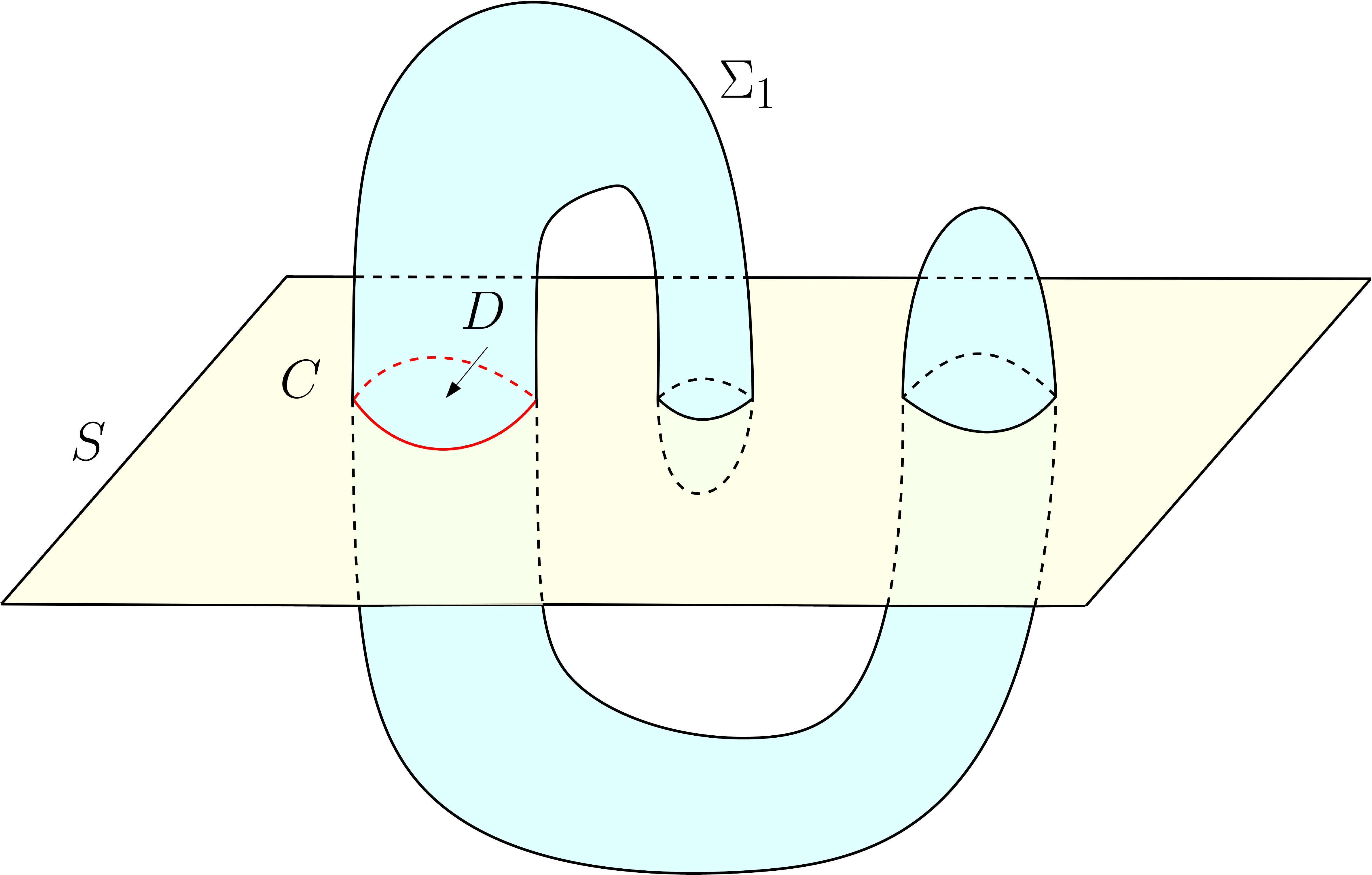} 
\end{minipage}\hfill
\begin{minipage}{0.5\textwidth}
\centering
\includegraphics[width=\textwidth]{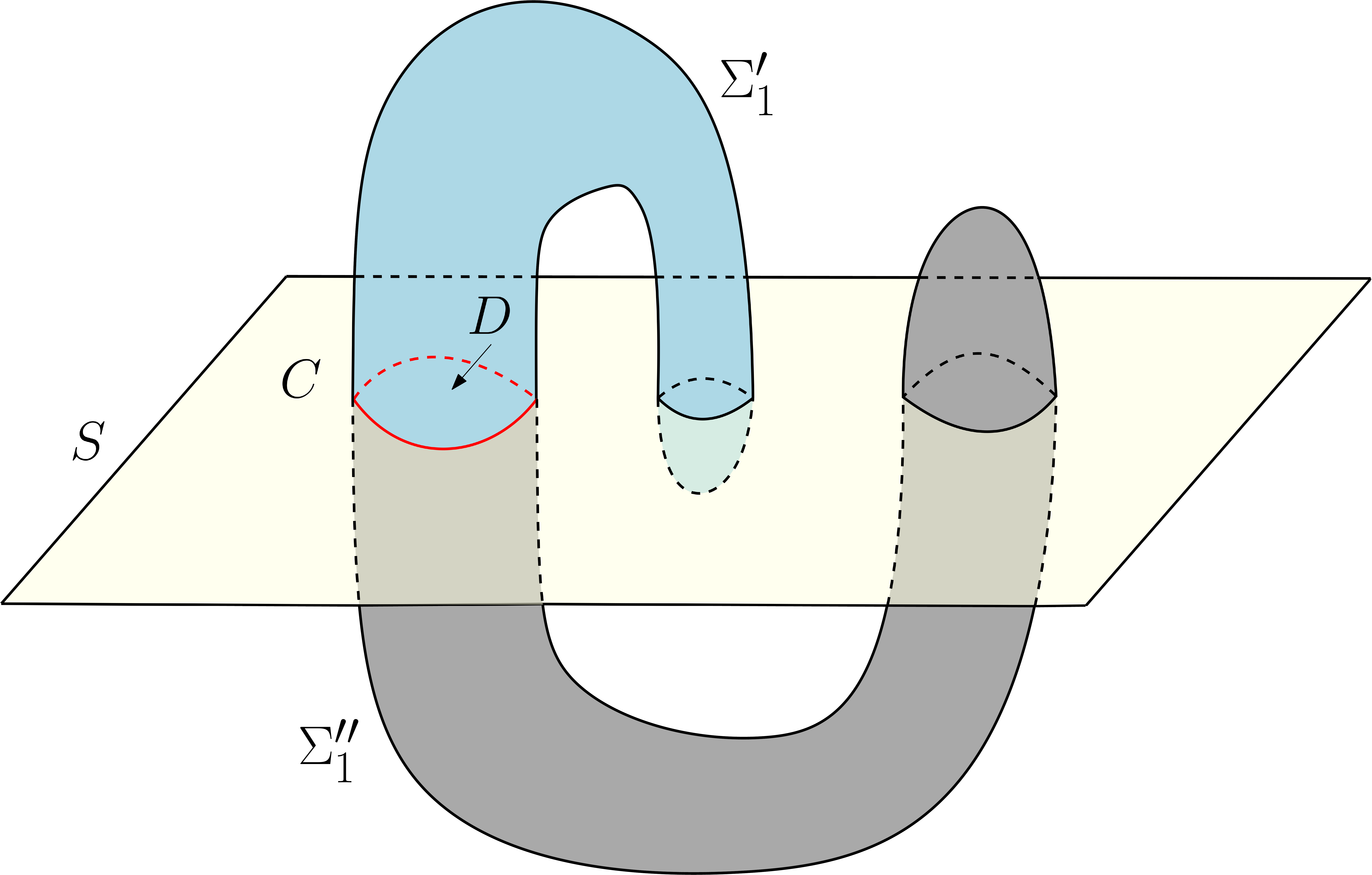} 
\end{minipage}
\caption{An illustration of the procedure to reduce the number of intersection curves between $S$ and $\Sigma$.}\label{fig:3}
\end{figure}

The next result is a less refined version of Proposition \ref{lem:separable}. The refinement is included in the original statement for the purposes of application to the setting in which portions of the boundary are either trapped or untrapped. Here, however, these considerations are not relevant.

\begin{lemma}\label{cor:B8}
Let $\Omega$ be as in Definition \ref{def:Sigma}, and assume that it has at least two boundary components. Suppose that $\partial_1\Omega$ has positive genus, and $\Omega$ satisfies the homotopy condition with respect to $\partial_1\Omega$. Then there exists $i_0\in \llbracket 2,n \rrbracket$ such that $\partial_1\Omega$ and $\partial_{i_0}\Omega$ are not separable by a 2-sphere.
\end{lemma}

\begin{proof}
Assume the conclusion is false, that is, $\partial_1\Omega$ and $\partial_i\Omega$ are separable by a 2-sphere for all $i\in \llbracket 2,n \rrbracket$. Let $\Sigma$ be a reduction system and let $j_1=j_\Sigma(1)$. By Lemma~\ref{cor:B7}, we have that $j_\Sigma(i)\neq j_1$ for all $i\in\llbracket 2,n \rrbracket$. Therefore $\partial Q_{j_1}$ consists of $\partial_1\Omega$ and 2-spheres, so that
\begin{equation}
0=[\partial Q_{j_1}]=[\partial_1\Omega ]+[\textup{spheres}]
\end{equation}
in $H_2(Q_{j_1})$. Moreover, $Q_{j_1}$ satisfies the homotopy condition with respect to $\partial_1\Omega$. Thus, there exists a continuous map $\rho:Q_j \rightarrow\partial_1\Omega$ such that its restriction to $\partial_1 \Omega$ is homotopic to the identity. It follows that
\begin{equation}
0=\rho_*\big( [\partial_1\Omega]+[\textup{spheres}] \big)=[\partial_1\Omega],
\end{equation}
which is impossible. In the second equality, we used the fact that every map from $S^2$ to a higher genus surface has degree zero. We conclude that there must exist $i_0\in \llbracket 2,n \rrbracket$ such that $\partial_1\Omega$ and $\partial_{i_0}\Omega$ are not separable by a 2-sphere.
\end{proof}

\begin{proof}[Proof of Proposition~\ref{lem:separable}]
This is a direct consequence of Lemma \ref{cor:B8}. To see this, simply note that the component $\partial_{i_0} \Omega$ which is not separable from $\partial_1^+ \Omega$ by a 2-sphere, must be among the $\partial_i^-\Omega$ since $\partial^+_i\Omega$ are 2-spheres for $i\neq 1$.
\end{proof}

\section{An Integral Identity}
\label{sec3}
\setcounter{equation}{0}
\setcounter{section}{3}

Spacetime harmonic functions satisfy a Bochner-type identity, which when integrated produces a natural relation between the dominant energy condition and the boundary geometry of initial data sets. This observation leads to a proof of the spacetime version of the positive mass theorem in the asymptotically flat and hyperboloidal settings \cite{BHKKZ,HKK}. Here we will present a version of the resulting integral identity suitable for the purposes of this paper. In particular, the boundary terms are analyzed in greater detail so that they may be related to the null expansions of the boundary.  The following result is a generalization of \cite[Proposition 3.2]{HKK} and \cite[Proposition 1.1]{HMT}. Although it is stated in full generality with arbitrary boundary conditions, the inequality will only be applied for constant Dirichlet boundary data, in which case several boundary integrals simplify.

\begin{proposition}\label{pro:div}
Let $(\Omega,g,k)$ be an orientable 3-dimensional compact initial data set with smooth boundary $\partial\Omega$, having outward unit normal $\mathbf{n}$. Let $u:\Omega\to\mathbb{R}$ be a spacetime harmonic function which lies in $C^{2,\alpha}(\Omega)$, $0<\alpha<1$, and denote the open subset of the boundary on which $|\nabla_{\partial} u|\neq 0$ by $\bar{\partial}\Omega$, where $\nabla_{\partial}u$ is the projection of the full gradient onto the boundary tangent space. The set of boundary points on which $|\nabla u|\neq 0$ will be labeled by $\hat{\partial}\Omega$. If $\overline{u}$ and $\underline{u}$ are the maximum and minimum values of $u$ and $\Sigma_t$ are $t$-level sets, then
\begin{align}\label{equ:prediv}
\begin{split}
&\int_{\hat{\partial}\Omega}\left(k(\nabla_{\partial}u,\mathbf{n})-|\nabla u|H-\mathbf{n}(u)\mathrm{Tr}_{\partial\Omega}k\right)dA\\
&+\int_{\bar{\partial} \Omega}\frac{|\nabla_{\partial}u|}{|\nabla u|}\nabla_{\partial}u\left(\frac{\mathbf{n}(u)}{|\nabla_{\partial}u|}\right) dA+2\pi\int_{\underline{u}}^{\bar{u}}\chi(\Sigma_t)dt \\
&\geq\int_{\Omega}\left(\frac{1}{2}\frac{|\bar{\nabla}^2 u|^2}{|\nabla u|}+\mu|\nabla u|+J(\nabla u)\right)dV,
\end{split}
\end{align}
where $\chi(\Sigma_t)$ is the Euler characteristic, $\bar{\nabla}^2 u$ is the spacetime Hessian, and $H$ is the mean curvature of the boundary with respect to $\mathbf{n}$.
\end{proposition}

\begin{proof}
The integral identity of \cite[Proposition 3.2]{HKK} states that
\begin{equation}\label{jghqk}
\int_{\hat{\partial}\Omega}\left(\mathbf{n}(|\nabla u|)+k(\nabla u,\mathbf{n})\right)dA\geq\int_{\underline{u}}^{\overline{u}}\int_{\Sigma_t}
\left(\frac{1}{2}\frac{|\bar{\nabla}^2 u|^2}{|\nabla u|^2}+\mu+J\left(\frac{\nabla u}{|\nabla u|}\right)-K\right)dA dt,
\end{equation}
where $K$ is the Gauss curvature of regular level sets $\Sigma_t$.
Observe that by Sard's theorem, the set of values in $[\underline{u},\overline{u}]$ which are critical for $u$ on $\Omega$ or $\partial\Omega$ is of measure zero; see
\cite[Remark 3.3]{HKK} for the applicability of Sard's theorem under the current regularity hypotheses. Thus, on the right-hand side of \eqref{jghqk} we may restrict attention to regular level sets $\Sigma_t$ for which $t$ is also a regular value of $u|_{\partial\Omega}$. These level sets intersect the boundary of $\Omega$ transversely, and hence
\begin{equation}
\partial\Sigma_t =\Sigma_t \cap\partial\Omega=\Sigma_t \cap \bar{\partial}\Omega
\end{equation}
consists of possibly multiple smooth closed curves in $\partial\Omega$.
We may then apply the Gauss-Bonnet theorem and coarea formula to find
\begin{align}\label{ajghiaifh}
\begin{split}
&\int_{\underline{u}}^{\overline{u}}\left(2\pi\chi(\Sigma_t)
-\int_{\Sigma_t \cap\bar{\partial}\Omega}\kappa\right)dt+
\int_{\hat{\partial}\Omega}\left(\mathbf{n}(|\nabla u|)+k(\nabla u,\mathbf{n})\right)dA\\
\geq & \int_{\Omega}\left(\frac{1}{2}\frac{|\bar{\nabla}^2 u|^2}{|\nabla u|}+\mu|\nabla u|+J(\nabla u)\right)dV,
\end{split}
\end{align}
where $\kappa$ denotes the geodesic curvature of $\Sigma_t \cap\bar{\partial}\Omega$ viewed as the boundary of the regular level set $\Sigma_t$.

The boundary terms of \eqref{ajghiaifh} will now be analyzed. Working on $\hat{\partial}\Omega$, a straightforward computation shows that
\begin{equation}\label{fjkaklgnaoo}
\mathbf{n}(|\nabla u|)=\frac{1}{|\nabla u|}\left(\mathbf{n}(u)\nabla_{\mathbf{n}}^2 u-II(\nabla_{\partial}u,\nabla_{\partial}u)
+\nabla_{\partial}u\left(\mathbf{n}(u)\right)\right),
\end{equation}
where $II(X,Y)=\langle\nabla_{X}\mathbf{n},Y\rangle$ with $X,Y\in T\partial\Omega$ denotes the second fundamental form of the boundary. Furthermore, as is shown below, on $\bar{\partial}\Omega$ the second fundamental form term may be expanded as
\begin{equation}\label{jshgkah}
-|\nabla u|^{-1} II(\nabla_{\partial}u,\nabla_{\partial}u)=|\nabla_{\partial}u|\kappa
-|\nabla u| H-\mathbf{n}(u)\mathrm{Tr}_{g}k-\frac{\mathbf{n}(u)}{|\nabla u|}\nabla_{\mathbf{n}}^2 u-\frac{\mathbf{n}(u)}{|\nabla_{\partial}u||\nabla u|}\nabla_{\partial}u(|\nabla_{\partial}u|).
\end{equation}
It follows that on $\bar{\partial}\Omega$ we have
\begin{equation}\label{aohaqiofgh}
\mathbf{n}(|\nabla u|)= |\nabla_{\partial}u|\kappa-|\nabla u|H- \mathbf{n}(u)\mathrm{Tr}_g k
+\frac{|\nabla_{\partial}u|}{|\nabla u|}\nabla_{\partial}u\left(\frac{\mathbf{n}(u)}{|\nabla_{\partial}u|}\right).
\end{equation}
Consider now the set $\hat{\partial}\Omega\setminus\bar{\partial}\Omega$, that is, boundary points where $|\nabla u|\neq 0$ but $|\nabla_{\partial}u|=0$. In this case, \eqref{fjkaklgnaoo} and the spacetime harmonic equation \eqref{shf} imply that
\begin{equation}\label{q9hgns}
\mathbf{n}(|\nabla u|)=\frac{\mathbf{n}(u)}{|\nabla u|}\nabla_{\mathbf{n}}^2 u =-\frac{\mathbf{n}(u)}{|\nabla u|}\left(H\mathbf{n}(u)+\left(\mathrm{Tr}_g k\right)|\nabla u|\right)=-|\nabla u|H- \mathbf{n}(u)\mathrm{Tr}_g k
\end{equation}
at points with $\Delta_{\partial}u=0$, where $\Delta_{\partial}$ is the Laplace-Beltrami operator with respect to the boundary metric. Moreover, the set of points with $|\nabla_{\partial}u|=0$ and $\Delta_{\partial}u\neq 0$ is of measure zero in $\partial\Omega$, as may be seen by applying the regular value theorem to the appropriate projection of $\nabla_{\partial}u$. Therefore combining \eqref{ajghiaifh}, \eqref{aohaqiofgh}, and \eqref{q9hgns} along with the coarea formula on the boundary, yields the desired result.

It remains to verify \eqref{jshgkah}. Each point of $\bar{\partial}\Omega$ lies on a smooth curve $\partial\Sigma_t $, for some level set $\Sigma_t$. We may then construct an orthogonal frame $\{\tau,\nu, \tilde{\mathbf{n}}\}$ at each such boundary point where: $\tau$ is the unit tangent vector to the curve, $\nu=\tfrac{\nabla u}{|\nabla u|}$, and $\tilde{\mathbf{n}}=\mathbf{n}-\tfrac{\mathbf{n}(u)}{|\nabla u|}\nu$ is the projection of the unit outer normal for $\partial\Omega$ onto the tangent space of $\Sigma_t$. The mean curvature of $\partial\Omega$ and the geodesic curvature of $\partial\Sigma_t$ may then be expressed as
\begin{equation}
H=\langle\nabla_{\tau}\mathbf{n},\tau\rangle+|\nabla_{\partial}u|^{-2}
\langle\nabla_{\nabla_{\partial}u}\mathbf{n},\nabla_{\partial}u\rangle,
\end{equation}
\begin{equation}
\kappa=\left\langle\nabla_{\tau}
\frac{\tilde{\mathbf{n}}}{|\tilde{\mathbf{n}}|},\tau\right\rangle
=|\tilde{\mathbf{n}}|^{-1}\left(\langle\nabla_{\tau}\mathbf{n},\tau\rangle
-\frac{\mathbf{n}(u)}{|\nabla u|^2}\langle\nabla_{\tau}\nabla u,\tau\rangle\right).
\end{equation}
Therefore
\begin{equation}
-II(\nabla_{\partial}u,\nabla_{\partial}u)
=|\nabla_{\partial}u|^2\left(\langle\nabla_{\tau}\mathbf{n},\tau\rangle -H\right)
=|\nabla_{\partial}u|^2\left(|\tilde{\mathbf{n}}|\kappa
+\frac{\mathbf{n}(u)}{|\nabla u|^2}\langle\nabla_{\tau}\nabla u,\tau\rangle -H\right).
\end{equation}
Using the computation
\begin{equation}
|\tilde{\mathbf{n}}|^2 =1-\frac{\mathbf{n}(u)^2}{|\nabla u|^2}=\frac{|\nabla_{\partial}u|^2}{|\nabla u|^2}
\end{equation}
and the fact that $u$ is a spacetime harmonic function, we find that
\begin{align}\label{alfksdjfia}
\begin{split}
&-II(\nabla_{\partial}u,\nabla_{\partial}u)\\
=&\frac{|\nabla_{\partial}u|^3}{|\nabla u|}\kappa
-|\nabla_{\partial}u|^2 H
+\frac{\mathbf{n}(u)|\nabla_{\partial}u|^2}{|\nabla u|^2}\langle\nabla_{\tau}\nabla u,\tau\rangle\\
=&\frac{|\nabla_{\partial}u|^3}{|\nabla u|}\kappa
-|\nabla_{\partial}u|^2 H -\frac{\mathbf{n}(u)|\nabla_{\partial}u|^2}{|\nabla u|^2}\left(\nabla_{\mathbf{n}}^2 u+\frac{\nabla^2 u(\nabla_{\partial}u,\nabla_{\partial}u)}{|\nabla_{\partial}u|^2}
+\left(\mathrm{Tr}_g k\right)|\nabla u|\right).
\end{split}
\end{align}
Lastly, inserting
\begin{align}
\begin{split}
\nabla^2 u(\nabla_{\partial}u,\nabla_{\partial}u)
=&\nabla^2_{\partial}u(\nabla_{\partial}u,\nabla_{\partial}u)+
II(\nabla_{\partial}u,\nabla_{\partial}u)\mathbf{n}(u)\\
=&|\nabla_{\partial}u|\nabla_{\partial}u\left(|\nabla_{\partial}u|\right)
+II(\nabla_{\partial}u,\nabla_{\partial}u)\mathbf{n}(u)
\end{split}
\end{align}
into \eqref{alfksdjfia} and solving for $II(\nabla_{\partial}u,\nabla_{\partial}u)$ produces formula \eqref{jshgkah}.
\end{proof}

The boundary terms in the integral identity \eqref{equ:prediv} motivate a boundary value problem for the spacetime harmonic function. In particular, consider the case in which the function $u$ takes constant values on each connected component of $\partial\Omega$. This implies that $\bar{\partial}\Omega=\varnothing$, and $k(\nabla_{\partial}u,\mathbf{n})=0$, as well as $\mathbf{n}(u)=\pm |\nabla u|$. Thus, if additionally the sign of $\mathbf{n}(u)$ were prescribed appropriately, then the null expansions would appear in the boundary integrals. It turns out that this can be achieved by choosing the constants on the various components of the boundary correctly.

Consider the setting of Theorem \ref{thm:EGM}, in which the boundary is decomposed into a disjoint union
\begin{equation}\label{bdryd}
\partial\Omega=\big(\sqcup_{i=1}^m\partial^{+}_{i}\Omega\big)\sqcup \big(\sqcup_{i=1}^\ell\partial^-_{i}\Omega\big),
\end{equation}
where the connected components are organized so that $\theta_{+}\left(\partial_i^+ \Omega\right)\geq 0$ with respect to the outer normal, and $\theta_{+}\left(\partial_i^- \Omega\right)\leq 0$ with respect to the inner normal. The unit normal which takes this set of orientations at the various components will be denoted by $\upsilon$. We then propose the following boundary value problem which is closely related to that used in \cite{HKK}:
\begin{equation}\label{shf1}
\Delta u+\left(\mathrm{Tr}_g k\right)|\nabla u|=0\quad\quad\text{ on }\quad\quad\Omega,
\end{equation}
\begin{equation}\label{equ:boundarycondition1}
\left\{\begin{array}{cc}
u=1 &\ \textup{on}\ \partial_{1}^+\Omega,\\
u=0 &\ \textup{on}\ \partial_{1}^-\Omega,\\
u=a^+_i\in (0,1) &\ \textup{on}\ \partial^+_{i}\Omega,\ i\in\llbracket 2,m \rrbracket,\\
u=a^-_i\in (0,1) &\ \textup{on}\ \partial^-_{i}\Omega,\ i\in\llbracket 2,\ell \rrbracket,
\end{array} \right.
\end{equation}
and
\begin{equation}\label{equ:boundarycondition2}
\displaystyle\min_{\partial^+_{i}\Omega}\partial_{\upsilon}u=0 \text{ }\text{ for }\text{ } i\in\llbracket 2,m \rrbracket,\quad\quad\quad
\displaystyle\min_{\partial^-_{i}\Omega}\partial_{\upsilon}u=0  \text{ }\text{ for }\text{ } i\in\llbracket 2,\ell \rrbracket,
\end{equation}
for some constants $a_{i}^{\pm}$. For a solution as above, it holds that
\begin{equation}
\begin{split}
&\mathbf{n}(u)=\partial_{\upsilon}u=|\nabla u|\ \textup{ at }\ \partial^+_i\Omega\ \ \textup{ for }\  i\in\llbracket 1,m \rrbracket,\\
&\mathbf{n}(u)=-\partial_{\upsilon}u=-|\nabla u|\ \textup{ at }\ \partial^-_i\Omega\ \ \textup{ for }\  i\in\llbracket 1,\ell \rrbracket,
\end{split}
\end{equation}
where we have used the maximum principle for $i=1$. This implies that
\begin{equation}
\begin{split}
&-|\nabla u|H-\mathbf{n}(u)\textup{Tr}_{\partial \Omega}k=-\theta_+|\nabla u|\ \textup{ at }\ \partial^+_i\Omega\ \ \textup{ for }\  i\in\llbracket 1,m \rrbracket,\\
&-|\nabla u|H-\mathbf{n}(u)\textup{Tr}_{\partial \Omega}k=\theta_+|\nabla u|\ \textup{ at }\ \partial^-_i\Omega\ \ \textup{ for }\  i\in\llbracket 1,\ell \rrbracket,
\end{split}
\end{equation}
where the mean curvature in $\theta_+$ is computed with respect to $\upsilon$. Applying Proposition \ref{pro:div} with this spacetime harmonic function then yields
\begin{align}\label{aifjq0gn}
\begin{split}
&\int_{\Omega}\left(\frac{1}{2}\frac{|\bar{\nabla}^2 u|^2}{|\nabla u|}+\mu|\nabla u|+J(\nabla u)\right)dV
+\sum_{i=2}^{m}\int_{\hat{\partial}_i^+ \Omega}\theta_{+} |\nabla u|dA
-\sum_{i=1}^{\ell}\int_{\hat{\partial}_i^- \Omega}\theta_{+} |\nabla u|dA\\
\leq &2\pi\int_{\underline{u}}^{\bar{u}}\chi(\Sigma_t)dt-\int_{\hat{\partial}_1^+ \Omega}\theta_+ |\nabla u|dA.
\end{split}
\end{align}
Although the integral over $\hat{\partial}_1^+ \Omega$ is similarly nonnegative in this setting, we keep it in the expression \eqref{aifjq0gn} for later use with other hypotheses.

\section{The Spacetime Harmonic Function Boundary Value Problem}
\label{sec4}
\setcounter{equation}{0}
\setcounter{section}{4}

In this section we will solve the boundary value problem \eqref{shf1}, \eqref{equ:boundarycondition1} with auxiliary condition \eqref{equ:boundarycondition2}. Note that the auxiliary condition is needed only when $\Omega$ has more than two boundary components. The spacetime harmonic function equation admits a mild, effectively linear, nonlinearity.  This allows from a relatively straightforward application of the Leray-Schauder fixed point theorem, to establish existence for the Dirichlet problem, see \cite[Section 4.1]{HKK}. There it was shown that given a function $h\in C^{2,\alpha}(\partial\Omega)$, $\alpha\in (0,1)$ there is a unique solution $u\in C^{2,\alpha}(\Omega)$ of
\begin{equation}\label{Dirichlet-local}
\left\{ \begin{array}{cc}
\Delta u+ \left(\mathrm{Tr}_g k\right)|\nabla u|=0,  & \textup{ on }\ \Omega, \\
u=h, & \textup{ on }\ \partial\Omega,
\end{array} \right.
\end{equation}
satisfying the estimate
\begin{equation}\label{estimae}
\| u\|_{C^{2,\alpha}(\Omega)}\leq C\left(\alpha,\| h\|_{C^{2,\alpha}(\partial\Omega)}\right)
\end{equation}
where the constant $C$ also depends on $g$ and $k$ although this is not emphasized. Furthermore, the existence of constant boundary values $a_{i}^{\pm}$ for which the auxiliary condition \eqref{equ:boundarycondition2} is satisfied may be motivated as follows. Suppose that on one boundary component $\partial_i^+ \Omega$ the Dirichlet value is set to $a_i^+ =1$. Then by the maximum principle and the Hopf lemma, which applies to the spacetime harmonic function equation since the nonlinear first order part may be expressed as a linear term with bounded coefficients, we must have that the normal derivative satisfies $\mathbf{n}(u)>0$ on $\partial_i^+ \Omega$. Similarly, if we set $a_i^+ =0$ then $\mathbf{n}(u)<0$. Thus, if we vary the choice of $a_i^+$ from 1 to 0, while all other boundary values are held fixed, then there should be a value $a_i^+ \in (0,1)$ such that $\min_{\partial_i^+ \Omega}\mathbf{n}(u)=0$. It turns out that we are able to prove a slightly stronger result. In what follows, vectors in $\mathbb{R}^{m+\ell-2}$ will be denoted by $\vec{a}=(a^+_{2},\dots,a^+_{m},a^-_{2},\dots, a^-_{\ell})$, and we shall write $\vec{a}\leq\vec{b}$ if $a^+_i\leq b^+_i$ for all $ i\in\llbracket 2,m \rrbracket$ and $a^-_i\leq b^-_i$ for all $ i\in\llbracket 2,\ell \rrbracket$.

\begin{proposition}\label{lem:specialbd-app}
Let $(\Omega,g,k)$ be a smooth compact initial data set, with boundary satisfying the decomposition \eqref{bdryd}. Then there exists a unique vector $\vec{a}\in(0,1)^{m+\ell-2}$ and a unique function $u_{\vec{a}}\in C^{2,\alpha}(\Omega)$ satisfying \eqref{shf1}, \eqref{equ:boundarycondition1}, and \eqref{equ:boundarycondition2}. Furthermore, let $\vec{b}\in \mathbb{R}^{m+\ell-2}$ be any vector, and let $u_{\vec{b}}\in C^{2,\alpha}(\Omega)$ be the unique solution to \eqref{shf1} and \eqref{equ:boundarycondition1} with $\vec{a}$ replaced by $\vec{b}$. If $\mathbf{n}(u_{\vec{b}})\geq 0$ on $\partial\Omega\setminus\left(\partial_1^+ \Omega \cup \partial_1^- \Omega\right)$, where $\mathbf{n}$ is the unit outer normal, then $\vec{a}\leq\vec{b}$. In particular, $u_{\vec{a}}\leq u_{\vec{b}}$ on $\Omega$.
\end{proposition}

For consistency of orientation at the boundary, in this section we will solely make use of the unit outer normal $\mathbf{n}$ to $\partial\Omega$, instead of using the normal $\upsilon$. With this convention the auxiliary condition \eqref{equ:boundarycondition2} becomes
\begin{equation}\label{equ:boundarycondition99}
\displaystyle\min_{\partial^+_{i}\Omega}\mathbf{n}(u)=0 \text{ }\text{ for }\text{ } i\in\llbracket 2,m \rrbracket,\quad\quad\quad
\displaystyle\max_{\partial^-_{i}\Omega}\mathbf{n}(u)=0  \text{ }\text{ for }\text{ } i\in\llbracket 2,\ell \rrbracket.
\end{equation}
Next, we introduce some notation. Let $\vec{a}\in\mathbb{R}^{m+\ell-2}$ and consider the spacetime harmonic function $u_{\vec{a}}$ that satisfies the boundary conditions \eqref{equ:boundarycondition1}. Define a map
\begin{equation}
\Phi:\mathbb{R}^{m+\ell-2}\to \prod_{i=2}^m C^{1,\alpha}(\partial^+_{i}\Omega)\times \prod_{i=2}^\ell C^{1,\alpha}(\partial^-_{i}\Omega)
\end{equation}
given by
\begin{equation}\label{def:phi}
\Phi[\vec{a}]\coloneqq (\phi^+_2[\vec{a}],\dots,\phi^+_m[\vec{a}],
\phi^-_2[\vec{a}],\dots,\phi^-_\ell[\vec{a}]),
\end{equation}
where
\begin{equation}
\phi^+_i[\vec{a}]\coloneqq \mathbf{n}(u_{\vec{a}})\bigg|_{\partial^+_{i}\Omega}\text{ }\text{ for }\text{ } i\in\llbracket 2,m \rrbracket,\quad\quad\quad
\phi^-_i[\vec{a}]\coloneqq \mathbf{n}(u_{\vec{a}})\bigg|_{\partial^-_{i}\Omega} \text{ }\text{ for }\text{ } i\in\llbracket 2,\ell \rrbracket.
\end{equation}

\begin{lemma}\label{lem:phiconti}
The map $\Phi$ is continuous.
\end{lemma}
	
\begin{proof}
Let $\vec{a}_j\in \mathbb{R}^{m+\ell-2}$, $j\in\mathbb{N}$ be a sequence of vectors which converges to $\vec{a}_{\infty}$. The estimate \eqref{estimae} guarantees that the corresponding spacetime harmonic functions $u_{\vec{a}_j}$ are uniformly bounded in $C^{2,\beta}(\Omega)$ for any $\beta\in (\alpha,1)$. Observe that the Arzel\`a–Ascoli theorem yields the existence of a subsequence, still denoted by $u_{\vec{a}_j}$, which converges in $C^{2,\alpha}(\Omega)$. The limit is a spacetime harmonic function with boundary data given by $\vec{a}_{\infty}$. Since solutions to the Dirichlet problem \eqref{Dirichlet-local} are unique, we must have that the
limit agrees with $u_{\vec{a}_{\infty}}$. Therefore $u_{\vec{a}_j}$ converges to $u_{\vec{a}_{\infty}}$ in $C^{2,\alpha}(\Omega)$. This implies that $\phi^+_i[\vec{a}_j]$ converges to $\phi^+_i[\vec{a}_{\infty}]$ in $C^{1,\alpha}(\partial^+_{i}\Omega)$ for $i\in \llbracket 2,m \rrbracket$, and $\phi^-_i[\vec{a}_j]$ converges to $\phi^-_i[\vec{a}_{\infty}]$ in $C^{1,\alpha}(\partial^-_{i}\Omega)$ for $i\in \llbracket 2,\ell \rrbracket$.
\end{proof}
	
In order to facilitate the manipulation of boundary data, we introduce the following operations for $\vec{a}\in\mathbb{R}^{m+\ell-2}$ and $b\in\mathbb{R}$ in which the entry $a^+_i$ or $a^-_i$ is replaced by $b$, namely
\begin{equation}\label{def:pi}
\begin{split}
\pi^+_i(b,\vec{a})=&(a^+_2,\dots,a^+_{i-1},b,a^+_{i+1},\dots,a^+_m,a^-_2,
\dots,a^-_\ell),\quad\text{ } i\in \llbracket 2,m \rrbracket,\\
\pi^-_i(b,\vec{a})=&(a^+_2,\dots,a^+_m,a^-_2,\dots,a^-_{i-1},b,a^-_{i+1},
\dots,a^-_\ell),\quad\text{ } i\in \llbracket 2,\ell \rrbracket.
\end{split}
\end{equation}
Furthermore, consider the optimal values for the Dirichlet data of each boundary component, that is
\begin{equation}\label{def:b+-}
\begin{split}
T^+_{i}(\vec{a})\coloneqq & \inf\{b\in\mathbb{R}\,|\, \min_{\partial^+_i\Omega } \phi^+_i[\pi^+_i(b,\vec{a})] \geq 0\}\ \quad\textup{for}\ \quad i\in \llbracket 2,m \rrbracket,\\
T^-_{i}(\vec{a})\coloneqq & \inf\{b\in\mathbb{R}\,|\, \max_{\partial^-_i\Omega } \phi^-_i[\pi^-_i(b,\vec{a})] \geq 0\}\ \quad\textup{for}\ \quad i\in \llbracket 2,\ell \rrbracket.
\end{split}
\end{equation}
Notice that $T^+_{i}(\vec{a})$ does not depend on $a^+_i$ and $T^-_{i}(\vec{a})$ does not depend on $a^-_i$. Furthermore, the maximum principle combined with the Hopf lemma shows that the sets used in the definition of \eqref{def:b+-} are non-empty and bounded from below. Therefore $T^+_{i}(\vec{a})$ and $T^-_{i}(\vec{a})$ are finite. The next result collects the essential properties of these quantities.

\begin{lemma}\label{lem:bprop}
\hfill
\begin{enumerate}
\item Fix $i\in \llbracket 2,m \rrbracket$ and let $c\in\mathbb{R}$. Then
$\min_{\partial^+_i\Omega} \phi^+_i[\pi^+_i(c,\vec{a})]$ is positive, zero, or negative when $c>T^+_{i}(\vec{a})$, $c=T^+_{i}(\vec{a})$, or $c<T^+_{i}(\vec{a})$ respectively.

\item Fix $i\in \llbracket 2,\ell \rrbracket$ and let $c\in\mathbb{R}$. Then
$\max_{\partial^-_i\Omega} \phi^-_i[\pi^-_i(c,\vec{a})]$ is positive, zero, or negative when $c>T^-_{i}(\vec{a})$, $c=T^-_{i}(\vec{a})$, or $c<T^-_{i}(\vec{a})$ respectively.
			
\item If $\vec{a},\vec{b} \in\mathbb{R}^{m+\ell-2}$ with $\vec{a}\leq\vec{b}$ then
$T^+_{i}(\vec{a})\leq T^+_{i}(\vec{b})$ for $i\in\llbracket 2,m \rrbracket$, and
$T^-_{i}(\vec{a})\leq T^-_{i}(\vec{b})$ for $i\in\llbracket 2,\ell \rrbracket$.

\item The functions $T^+_{i}:\mathbb{R}^{m+\ell-2}\rightarrow\mathbb{R}$, $i\in\llbracket 2,m \rrbracket$ and $T^-_{i}:\mathbb{R}^{m+\ell-2}\rightarrow\mathbb{R}$, $i\in\llbracket 2,\ell \rrbracket$ are continuous.
\end{enumerate}
\end{lemma}

\begin{proof}
We begin with $(1)$. By Lemma \ref{lem:phiconti} the quantity $\phi^+_{i}[\vec{a}]$ depends continuously on $\vec{a}$, and therefore $\min_{\partial^+_i\Omega} \phi^+_i[\pi^+_i( {c} ,\vec{a})] $ is a continuous function of $c$. In order to establish $(1)$, it suffices to show that for fixed $\vec{a}$, the function $c\mapsto\min_{ \partial^+_i\Omega} \phi^+_i[\pi^+_i({c} ,\vec{a})]$ is strictly increasing. To this end let $c<\tilde{c}$, and consider the spacetime harmonic functions $u_{c,\vec{a}}$ and $u_{\tilde{c},\vec{a}}$ satisfying the boundary conditions \eqref{equ:boundarycondition1} with $a^+_i$ replaced by $c$ and $\tilde{c}$, respectively. A direct computation shows that $v\coloneqq u_{\tilde{c},\vec{a}}-u_{c,\vec{a}}$ solves the equation
\begin{equation}\label{alkjfja}
\Delta v +\left(\mathrm{Tr}_g k\right) \frac{\nabla (v +2u_{c,\vec{a}})}
{|\nabla (v +u_{c,\vec{a}})|+| \nabla u_{c,\vec{a}}|}\cdot\nabla v  =0\quad \textup{ on }\quad \Omega.
\end{equation}
Moreover, $v=0$ on $\partial\Omega\setminus \partial^+_{i}\Omega$ and $v=\tilde{c}-c>0$ on $ \partial^+_{i}\Omega$. Thus, by the Hopf lemma
$\mathbf{n}(v)>0$ on $\partial^+_{i}\Omega$. This implies that
\begin{equation}
\min_{\partial^+_{i}\Omega} \phi^+_i[\pi^+_i(\tilde{c} ,\vec{a})] =\min_{\partial^+_{i}\Omega}\left(\mathbf{n}(v)+\phi^+_i[\pi^+_i( {c} ,\vec{a})] \right)>\min_{\partial^+_{i}\Omega} \phi^+_i[\pi^+_i( {c} ,\vec{a})],
\end{equation}
which completes the proof of $(1)$. The proof of $(2)$ is similar and so we omit it.
		
Next, consider statement (3). Assume that $\vec{a},\vec{b}\in\mathbb{R}^{m+\ell-2}$ with $\vec{a}\leq\vec{b}$, let
$b\in\mathbb{R}$, and fix $i\in\llbracket 2,m \rrbracket$. Denote the spacetime harmonic functions satisfying the boundary conditions \eqref{equ:boundarycondition1} associated to $\pi_i^+(b,\vec{a})$ and $\pi_i^+(b,\vec{b})$, by $u_{b,\vec{a}}$ and $u_{b,\vec{b}}$, respectively.
As above, a computation shows that the function $\tilde{v}\coloneqq u_{b,\vec{b}}-u_{b,\vec{a}}$ solves an equation analogous to \eqref{alkjfja}.
Furthermore, $\tilde{v} \geq 0$ on $\partial\Omega$ and $\tilde{v}=0$ on
$\partial^+_{i}\Omega$. This implies that $ \mathbf{n}(\tilde{v})\leq 0$ on $\partial^+_{i}\Omega$, and thus
\begin{equation}
\min_{\partial^+_{i}\Omega} \phi^+_i[\pi^+_i(b,\vec{b})] =\min_{\partial^+_{i}\Omega}\left(\mathbf{n}(\tilde{v})+ \phi^+_i[\pi^+_i(b,\vec{a})]\right)\leq \min_{\partial^+_{i}\Omega} \phi^+_i[\pi^+_i(b,\vec{a})].
\end{equation}
Together with the monotonicity of $\min_{\partial^+_{i}\Omega} \phi^+_i[\pi^+_i(b,\vec{a})]$ in $b$, it follows that $T^+_{i}(\vec{a})\leq T^+_{i}(\vec{b})$. Similar arguments may be used to establish the remaining cases of $(3)$.

Lastly, we address $(4)$. Let $\vec{a}_j \in\mathbb{R}^{m+\ell-2}$ be a sequence of vectors which converges to $\vec{a}_{\infty}$. To prove that $T^+_i$ is continuous, it suffices to show that
\begin{equation}\label{equ:limsup}
\limsup_{j\to\infty} T^+_{i}(\vec{a}_j) \leq T^+_i(\vec{a}_{\infty}), \quad\text{ and }\quad
\liminf_{j\to\infty} T^+_{i}(\vec{a}_j) \geq T^+_i(\vec{a}_{\infty}).
\end{equation}
Suppose that the first inequality of \eqref{equ:limsup} fails. Then up to a subsequence, there exists $\epsilon>0$ such that $T^+_{i}(\vec{a}_j)\geq T^+_i(\vec{a}_{\infty})+\epsilon$. Therefore
\begin{align}\label{aoagjhos}
\begin{split}
0=\lim_{j\to\infty}\min_{\partial^+_i\Omega }\phi^+_i[\pi^+_i( T^+_{i}(\vec{a}_j) ,\vec{a}_j)]  &\geq \lim_{j\to\infty}\min_{\partial^+_i\Omega}\phi^+_i[\pi^+_i( T^+_i(\vec{a}_{\infty})+\epsilon ,\vec{a}_j)]\\
&=\min_{\partial^+_i\Omega }\phi^+_i[\pi^+_i( T^+_i(\vec{a}_{\infty})+\epsilon ,\vec{a}_{\infty})],
\end{split}
\end{align}
where we have used monotonicity of the map $c\mapsto\min_{\partial^+_i\Omega}\phi^+_i[\pi^+_i( {c} ,\vec{a})]$, as well as Lemma \ref{lem:phiconti}. Furthermore, since the monotonicity is strict it follows that the right-hand side of \eqref{aoagjhos} is strictly positive, which leads to a contradiction. We conclude that the first inequality of \eqref{equ:limsup} holds. The second inequality of \eqref{equ:limsup} may be dealt with similarly, and thus the continuity of $T^+_i$ is established. The continuity of $T^-_i$ can be proved in an analogous way.
\end{proof}

We now have the tools required to establish Proposition~\ref{lem:specialbd-app}. This result may be reformulated in a concise manner through the use of  notation developed in this section, together with the map $\vec{T}  :\mathbb{R}^{m+\ell-2}\to\mathbb{R}^{m+\ell-2}$ defined by
\begin{equation}\label{def:b}
\vec{T}(\vec{a})=(T^+_{2}(\vec{a}),\dots,T^+_{m}(\vec{a}),T^-_{2}(\vec{a}),\dots, T^-_{\ell}(\vec{a})).
\end{equation}

\begin{proposition}\label{lem:specialbd-T}
There exists a unique $\vec{a}\in(0,1)^{m+\ell-2}$ such that $\vec{T}(\vec{a})=\vec{a}$.
Furthermore, if $\vec{b}\in\mathbb{R}^{m+\ell-2}$ has the property that $\phi^+_i[\vec{b}]\geq 0$ on $\partial^+_{i}\Omega$ for all $i\in\llbracket 2,m \rrbracket $ and $\phi^-_i[\vec{b}]\geq 0$ on $\partial^-_{i}\Omega$ for all $i\in\llbracket 2,\ell \rrbracket $, then $\vec{a}\leq \vec{b}$. In particular, $u_{\vec{a}}\leq u_{\vec{b}}$ on $\Omega$.
\end{proposition}

\begin{proof}
Consider first the existence of $\vec{a}$. Let $\vec{a}_0=(1,1,\dots, 1)$ and inductively define $\vec{a}_{j+1}=\vec{T} (\vec{a}_j)$.  By the Hopf lemma, $\phi^+_i[\vec{a}_0 ]\geq 0$ on $\partial^+_{i}\Omega\ \textup{for}\ i\in\llbracket 2,m \rrbracket$ and $\phi^-_i[\vec{a}_0 ]\geq 0$ on $\partial^-_{i}\Omega\ \textup{for}\ i\in\llbracket 2,\ell \rrbracket$. Moreover, according to the definition of $\vec{T}$ we must have $\vec{a}_1 \leq \vec{a}_0$. Inductively applying part $(3)$ of Lemma~\ref{lem:bprop} then shows that the components of $\vec{a}_j$ each form a monotone non-increasing sequence. Furthermore, an inductive application of the Hopf lemma shows that each component of $\vec{a}_j$ is non-negative. Therefore, $\vec{a}_j$ converges to a limit $\vec{a}$. No component of $\vec{a}$ can be 0 or 1, which again follows from the Hopf lemma, and thus $\vec{a}\in (0,1)^{m+\ell-2}$. Lastly, part $(4)$ of Lemma~\ref{lem:bprop} states that $\vec{T}$ is continuous, and hence
$\vec{T}(\vec{a})=\vec{a}$.
		
Next, we prove uniqueness of the solution $\vec{a}$. Suppose there exists another fixed point $\vec{b}\neq\vec{a}$. Recall that $u_{\vec{a}}$ and $u_{\vec{b}}$ are the unique spacetime harmonic functions with Dirichlet boundary conditions determined by $\vec{a}$ and $\vec{b}$, from \eqref{equ:boundarycondition1}. We may assume without loss of generality that there is a component of $\vec{b}$ which is strictly larger than the corresponding component of $\vec{a}$, otherwise the roles of $\vec{b}$ and $\vec{a}$ may be reversed in the following argument. Observe that the maximum of $u_{\vec{b}}-u_{\vec{a}}$ must be achieved on $\partial\Omega\setminus \left(\partial_1^+ \Omega\cup\partial_1^- \Omega\right)$, since this function satisfies an equation of the form \eqref{alkjfja}. If the maximum is achieved on $\partial^+_i\Omega$, for some $i\in\llbracket 2,m \rrbracket$, then the Hopf lemma applied to $u_{\vec{b}}-u_{\vec{a} }$ implies that
\begin{equation}
\phi^+_i[\vec{b}]=\mathbf{n}(u_{\vec{b}})\bigg|_{\partial^+_{i}\Omega} >\mathbf{n}(u_{\vec{a}})\bigg|_{\partial^+_{i}\Omega}=\phi^+_i[\vec{a}].
\end{equation}
This, however, contradicts the fact that
\begin{equation}
\min_{\partial^+_{i}\Omega} \phi^+_i[\vec{b} ]=\min_{\partial^+_{i}\Omega}
\phi_i[\vec{a}]=0.
\end{equation}
A similar argument holds if the maximum is achieved on $\partial^-_i\Omega$, for some $i\in\llbracket 2,\ell \rrbracket$. Therefore, $\vec{a}$ is unique.
		
Now suppose that $\vec{b}\in\mathbb{R}^{m+\ell-2}$ satisfies  $\phi^+_i[\vec{b}]\geq 0$ on $\partial^+_{i}\Omega$ for all $i\in\llbracket 2,m \rrbracket $ and $\phi^-_i[\vec{b}]\geq 0$ on $\partial^-_{i}\Omega$ for all $i\in\llbracket 2,\ell \rrbracket $. The Hopf lemma shows that each component of $\vec{b}$ is strictly positive. By choosing $\vec{a}_0=\vec{b}$ and repeating the above iteration procedure, we find that there exists a fixed point solution $\vec{a}_{\infty}\in (0,1)^{m+\ell-2}$ with $\vec{a}_{\infty}\leq \vec{b}$. According to the uniqueness of such fixed points proven above, it follows that $\vec{a}=\vec{a}_{\infty}\leq \vec{b}$. Finally, the maximum principle shows that $u_{\vec{a}}\leq u_{\vec{b}}$ on $\Omega$.
\end{proof}

\section{Proof of Theorem~\ref{thm:EGM}}
\label{sec5}
\setcounter{equation}{0}
\setcounter{section}{5}

Let $(\Omega,g,k)$ be a smooth orientable 3-dimensional compact initial data set with boundary $\partial\Omega$, satisfying the dominant energy condition and $H_2(\Omega,\tilde{\partial} \Omega;\mathbb{Z})=0$. Suppose that the boundary may be decomposed into a disjoint union
\begin{equation}
\partial\Omega=\big(\sqcup_{i=1}^m\partial^{+}_{i}\Omega\big)\sqcup \big(\sqcup_{i=1}^\ell\partial^-_{i}\Omega\big),
\end{equation}
where the connected components are organized so that $\theta_{+}\left(\partial_i^+ \Omega\right)\geq 0$ with respect to the outer normal, and $\theta_{+}\left(\partial_i^- \Omega\right)\leq 0$ with respect to the inner normal. Moreover, assume that $\partial^+_{1}\Omega$ has positive genus, that $\partial^+_{i}\Omega$ is of zero genus for $i=2,\ldots,m$, and that $\Omega$ satisfies the homotopy condition with respect to $\partial^+_{1}\Omega$. By Proposition \ref{lem:separable}, it may be assumed that the ordering of $\partial_i^- \Omega$, $i\in \llbracket 1,\ell \rrbracket$ has been arranged so that $\partial_1^- \Omega$ is not separable from $\partial_1^+ \Omega$ by a 2-sphere. Next, let $u\in C^{2,\alpha}(\Omega)$ be the unique solution of \eqref{shf1}, \eqref{equ:boundarycondition1}, and \eqref{equ:boundarycondition2} given by Proposition \ref{lem:specialbd-app}.
We may then apply Proposition \ref{pro:div} and the discussion of Section \ref{sec3}, in particular \eqref{aifjq0gn}, to find
\begin{align}\label{aifjq0gn124}
\begin{split}
\int_{\Omega}\left(\frac{1}{2}\frac{|\bar{\nabla}^2 u|^2}{|\nabla u|}+\mu|\nabla u|+J(\nabla u)\right)dV \leq & 2\pi\int_{0}^{1}\chi(\Sigma_t)dt\\
&+\sum_{i=1}^{\ell}\int_{\partial_i^- \Omega}\theta_{+} |\nabla u|dA
-\sum_{i=1}^{m}\int_{\partial_i^+ \Omega}\theta_{+} |\nabla u|dA.
\end{split}
\end{align}

We will now show that the level set Euler characteristics satisfy $\chi(\Sigma_t)\leq 0$ for all regular values $t\in[0,1]$. This will be a consequence of the special boundary conditions chosen for $u$, the vanishing second relative homology, and the fact that $\Omega$ satisfies the homotopy condition with respect to a surface of positive genus. Let $\Sigma_t$ be a regular level set for $t\neq 0,1$. It suffices to show that $\chi(\Sigma'_t)\leq 0$, for an arbitrary connected component $\Sigma'_t$ of $\Sigma_t$. Note that $\Sigma'_{t}$ is a 2-sided properly embedded submanifold, which does not intersect $\partial\Omega$ in light of the boundary conditions chosen for $u$. Thus, we need only show that $\Sigma'_t$ is not a 2-sphere.
	
Proceeding by contradiction, let us suppose that ${\Sigma}'_t$ is indeed a 2-sphere. Since the second homology relative to certain boundary components vanishes, $H_2(\Omega;\mathbb{Z})$ is generated by boundary cycles and hence there exist $c^+_{i}$, $c^-_{i}\in\mathbb{Z}$ such that
\begin{equation}\label{equ:generator}
[{\Sigma}'_t]+\sum_{i=1}^m c^+_{i} [\partial^+_{i}\Omega]+\sum_{i=1}^\ell c^-_{i} [\partial^-_{i}\Omega]=0\qquad \text{ in }\qquad H_2(\Omega;\mathbb{Z}).
\end{equation}
Let $\tilde{\Omega}$ be the compact manifold without boundary obtained by filling in 3-balls and handlebodies along $\partial\Omega$. Then $[\Sigma'_t]=0$ as an element in $H_2(\tilde{\Omega};\mathbb{Z})$. This implies that there exists a domain $\tilde{D}\subset\tilde{\Omega}$ such that $\Sigma'_t=\partial \tilde{D}$. Set $D=\tilde{D}\cap \Omega$. Then the boundary of $D$ consists of $\Sigma'_t$ and some (possibly empty) connected components of $\partial\Omega$. Therefore, by changing the orientation of $\Sigma'_t$ if necessary, the coefficients in \eqref{equ:generator} are either $1$ or $0$. Moreover, since the sum of all boundary cycles is trivial,
we may assume that $c^+_{1}=0$ by further changing the orientation of $\Sigma'_t$ as needed. In fact, it must also be the case that $c_1^- =0$
because $\partial_1^- \Omega$ is not separable from $\partial_1^+ \Omega$ by a 2-sphere. Therefore, there exist index sets $I\subset \llbracket 2,m\rrbracket$ and $J\subset \llbracket 2,\ell\rrbracket$ such that
\begin{equation}
[{\Sigma}'_t]+\sum_{i\in I}[\partial^+_{i}\Omega]+\sum_{i\in J}[\partial^-_{i}\Omega]=0\qquad \text{ in }\qquad H_2(\Omega;\mathbb{Z}),
\end{equation}
and $\partial D={\Sigma}'_t\sqcup (\sqcup_{i\in I}\partial^+_{i}\Omega)\sqcup (\sqcup_{i\in J}\partial^-_{i}\Omega)$. By the maximum principle, the maximum and minimum of $u$ on $D$ must be achieved on $\partial D$. However, the Hopf lemma together with the boundary condition \eqref{equ:boundarycondition2} show that neither of these extrema can occur on $\partial D\cap\partial\Omega$. It follows that both maximum and minimum are obtained on $\Sigma'_t$, and hence $u$ is constant within $D$. This contradicts the regularity of $\Sigma'_t$ as a level set. We conclude that $\Sigma'_t$ cannot be a 2-sphere, and thus must have nonpositive Euler characteristic.

Consider now the case when $t=0,1$. Note that $\chi(\Sigma_1)\leq 0$ by assumption, since $\Sigma_1=\partial_1^+ \Omega$ is taken to have positive genus. Furthermore, by the Hopf lemma both of these are regular values for $u$. Therefore, a small neighborhood of $\partial_1^{-} \Omega$ is foliated by regular level sets $\Sigma_t$, $t>0$ all having the same topology as this boundary component. Since the Euler characteristic of these level sets is nonpositive, the same is true for the boundary component: $\chi(\Sigma_0)\leq 0$.

According to \eqref{aifjq0gn124}, the dominant energy condition, the sign constraint on the null expansions of the boundary, and the observation concerning Euler characteristics of level sets imply that
\begin{equation}\label{alfhi39}
\theta_+ |\nabla u|=0 \quad\text{ on }\quad \partial\Omega,\quad\quad\quad\quad
\chi(\Sigma_t)=0 \quad \text{ for all regular values } \quad t\in[0,1],
\end{equation}
\begin{equation}\label{equ:K=0}
|\nabla^2 u+|\nabla u|k|=0,\quad\quad\quad  |\nabla u|\mu+J(\nabla u)=0,\quad \text{ on }\quad \Omega.
\end{equation}
It follows from the first equation of \eqref{equ:K=0} that whenever $|\nabla u|\neq 0$ we have
\begin{equation}
|\nabla \log|\nabla u||\leq\frac{|\nabla^2 u|}{|\nabla u|}\leq \sup_{\Omega}|k|.
\end{equation}
Thus, applying this estimate along curves emanating from $\partial_1^{\pm}\Omega$, where $|\nabla u|>0$ by the Hopf lemma, shows that $|\nabla u|>0$ on all of $\Omega$. In particular, this is incompatible with the boundary condition \eqref{equ:boundarycondition2}, so there can be only two boundary components and $m=\ell=1$. Moreover, since all level sets are regular and have vanishing Euler characteristic we find that $\Sigma_t \cong T^2$, $\Omega\cong [0,1]\times T^2$, and the metric can be expressed as
\begin{equation}\label{matricllaf}
g=|\nabla u|^{-2}dt^2 +g_t
\end{equation}
for some family of metrics $g_t$ on the torus. In addition, the nonvanishing gradient together with the dominant energy condition and \eqref{equ:K=0}, imply that $\mu=|J|_g=-J(\nu)$ on $\Omega$ where $\nu=\tfrac{\nabla u}{|\nabla u|}$. In particular, the orthogonal projection of $J$ to any level set vanishes $J|_{\Sigma_t}=0$.

Next, note that $II_t=\tfrac{\nabla^2 u}{|\nabla u|}|_{\Sigma_t}$ is the second fundamental form of the $t$-level set, and therefore the first equation of \eqref{equ:K=0} yields
\begin{equation}\label{gajaiogj}
0= \big( \nabla^2 u+|\nabla u|k\big)\big|_{\Sigma_t}= \left(II_t+k\big|_{\Sigma_t}  \right)|\nabla u|.
\end{equation}
This shows that the (future) null second fundamental forms vanish $\chi^+ =0$, and hence each level set $\Sigma_t$ is a MOTS with respect to $\nu$. Moreover, since $\partial_t =f^{-1} \nu$ with $f=|\nabla u|$, the first variation of null expansion formula \cite{AMS,G} gives
\begin{equation}\label{fjaijsroa}
0=\partial_t \theta_{+}=-\Delta_t f^{-1} +2\langle X,\nabla_t f^{-1}\rangle
+\left(K_t -\mu-J(\nu)-\frac{1}{2}|\chi^+|^2 +\mathrm{div}_t X-|X|^2\right)f^{-1},
\end{equation}
where $K_t$, $\Delta_t$, and $X=k(\nu,\cdot)$ are respectively the Gauss curvature, Laplace-Beltrami operator, and a 1-form on $\Sigma_t$. Multiplying by $f$ and integrating by parts, while utilizing the Gauss-Bonnet theorem and the vanishing of the null second fundamental forms as well as the vanishing of the sum of energy and momentum densities, produces
\begin{equation}\label{=-09}
0=-\int_{\Sigma_t}\left(f\Delta_t f^{-1} +2\langle X,\nabla_t \log f\rangle +|X|^2\right)dA=-\int_{\Sigma_t}|\nabla_t \log f +X|^2 dA.
\end{equation}
It follows that $X=-\nabla_t \log f$, and thus from \eqref{fjaijsroa} we find $K_t \equiv 0$ so that $(\Sigma_t,g_t)$ is a flat torus for all $t\in[0,1]$.

Consider now the case in which $k=-\lambda g$ for some $\lambda\in C^{\infty}(\Omega)$. Since the momentum density vanishes when evaluated on vector fields $Y$ tangential to $\Sigma_t$, we have
\begin{equation}
0=J(Y)=\mathrm{div}_g\left(k-(\mathrm{Tr}_g k)g\right)(Y)=2Y(\lambda),
\end{equation}
so that $\lambda$ is constant on $\Sigma_t$ and we may write $\lambda=\lambda(t)$. Next observe that the first equation of \eqref{equ:K=0} implies
\begin{equation}
Y(|\nabla u|)=\nabla^2 u\left( \frac{\nabla u}{|\nabla u|},Y \right)
=\lambda g(\nabla u,Y)=0,
\end{equation}
so that $|\nabla u|$ is a constant on $\Sigma_t$. Furthermore
\begin{equation}\label{fgkahjgowj}
\partial_t |\nabla u|=g\left( \nabla|\nabla u|, \frac{\nabla u}{|\nabla u|^2}\right)=\nabla^2 u\left( \frac{\nabla u}{|\nabla u|},\frac{\nabla u}{|\nabla u|^2} \right)=\lambda g(\nu,\nu)=\lambda.
\end{equation}
Define a new radial coordinate $s=s(t)$ such that $ds=|\nabla u|^{-1} dt$ and $s(0)=0$. Then \eqref{fgkahjgowj} shows $\partial_{s}\log f=\lambda$, and with the help of \eqref{gajaiogj} we find
\begin{equation}
\frac{1}{2}\partial_{s}g_{s}=II_{s}=-k\big|_{\Sigma_{s}}
=\lambda g_{s}
=\left(\partial_{s}\log f\right) g_{s}
\end{equation}
so that $g_{s} =f(s)^2 \hat{g}$ for some flat metric $\hat{g}$ on $T^2$. It follows from \eqref{matricllaf} that the desired form of the metric is achieved
\begin{equation}
g=ds^2 +f(s)^2 \hat{g}.
\end{equation}

\section{Proof of Theorems \ref{thm:PMT}-\ref{energylb}}
\label{sec6}
\setcounter{equation}{0}
\setcounter{section}{6}

Let $(M,g,k)$ be a smooth orientable 3-dimensional asymptotically hyperboloidal initial data set with toroidal infinity. Suppose that the boundary may be decomposed into a disjoint union
\begin{equation}\label{afksgia}
\partial M=\big(\sqcup_{i=2}^m\partial^{+}_{i}M\big)\sqcup \big(\sqcup_{i=1}^\ell\partial^-_{i}M\big),
\end{equation}
where the connected components are organized so that $\theta_{+}\left(\partial_i^+ M\right)\geq 0$ with respect to the outer normal, and $\theta_{+}\left(\partial_i^- M\right)\leq 0$ with respect to the inner normal; this unit normal having the stated orientations will be denoted by $\upsilon$. Moreover, assume that $\partial^+_{i}M$ is of zero genus for $i=2,\ldots,m$, that $M$ satisfies the homotopy condition with respect to conformal infinity, and $H_2(M,\partial M;\mathbb{Z})=0$. Note that the integers $m,\ell\geq 1$, with $m=1$ signifying that the first set of components in \eqref{afksgia} is empty. Thus, the boundary is nonempty and weakly trapped with respect to the unit normal pointing towards the asymptotic end, having at least one weakly outer trapped component and with each weakly inner trapped component of genus zero.

For each $r>1$, let $T_r$ denote the constant radial coordinate torus in the asymptotic end, and set $M_r$ to be the bounded component of $M\setminus T_r$. Its boundary is then given by
\begin{equation}
\partial M_r=\partial_{1}^{+}M_r \sqcup \big(\sqcup_{i=2}^m\partial^{+}_{i}M_r\big)\sqcup \big(\sqcup_{i=1}^\ell\partial^-_{i}M_r\big),
\end{equation}
where $\partial_{1}^{+}M_r = T_r$, $\partial_i^+ M_r=\partial_i^+ M$ for $i\neq 1$, and $\partial_i^- M_r =\partial_i^- M$ for all $i\in \llbracket 1,\ell \rrbracket$. By Proposition \ref{lem:separable}, it may be assumed that the ordering of $\partial_i^- M$ has been arranged so that $\partial_1^- M_r$ is not separable from $\partial_1^+ M_r$ by a 2-sphere. Furthermore, let $w\in C^{2,\alpha}(\bar{M}_r)$ be the unique solution of \eqref{shf1}, \eqref{equ:boundarycondition1}, and \eqref{equ:boundarycondition2} given by Proposition \ref{lem:specialbd-app} with $w=0$ on $\partial_1^- M_r$ and $w=1$ on $\partial_1^+ M_r$. Define $u_r =r w$, and observe that this function satisfies
\begin{equation}\label{shf11}
\Delta u_r+\left(\mathrm{Tr}_g k\right)|\nabla u_r|=0\quad\quad\text{ on }\quad\quad M_r,
\end{equation}
\begin{equation}\label{equ:boundarycondition11}
\left\{\begin{array}{cc}
u_r=r &\ \textup{on}\ \partial_{1}^+ M_r,\\
u_r=0 &\ \textup{on}\ \partial_{1}^- M_r,\\
u_r=r a^+_i(r) &\ \textup{on}\ \partial^+_{i} M_r,\ i\in\llbracket 2,m \rrbracket,\\
u_r=r a^-_i(r) &\ \textup{on}\ \partial^-_{i} M_r,\ i\in\llbracket 2,\ell \rrbracket,
\end{array} \right.
\end{equation}
and
\begin{equation}\label{equ:boundarycondition21}
\displaystyle\min_{\partial^+_{i} M_r}\partial_{\upsilon}u_r=0 \text{ }\text{ for }\text{ } i\in\llbracket 2,m \rrbracket,\quad\quad\quad
\displaystyle\min_{\partial^-_{i} M_r}\partial_{\upsilon}u_r=0  \text{ }\text{ for }\text{ } i\in\llbracket 2,\ell \rrbracket,
\end{equation}
for some constants $a_{i}^{\pm}(r)\in (0,1)$. We may then apply Proposition \ref{pro:div} and the discussion of Section \ref{sec3}, in particular \eqref{aifjq0gn}, to find
\begin{align}\label{aifjq0gn2856635r}
\begin{split}
&\int_{M_r}\left(\frac{1}{2}\frac{|\bar{\nabla}^2 u_r|^2}{|\nabla u_r|}+\mu|\nabla u_r|+J(\nabla u_r)\right)dV
-\sum_{i=2}^{m}\int_{\partial_i^+ M_r}\theta_{-} |\nabla u_r|dA
-\sum_{i=1}^{\ell}\int_{\partial_i^- M_r}\theta_{+} |\nabla u_r|dA\\
\leq &2\pi\int_{0}^{r}\chi(\Sigma_t^r)dt-\int_{\partial_1^+ M_r}\theta_+ |\nabla u_r|dA,
\end{split}
\end{align}
where the null expansions are computed with respect to the unit normal pointing towards the asymptotic end and $\Sigma_t^r$ denotes the $t$-level set of $u_r$.

We will show that the integral over $\partial_1^+ M_r$ converges to a positive multiple of the total energy, as $r\rightarrow\infty$. To accomplish this, we will first estimate the asymptotics for $u_r$ and its derivatives. In the next result, suitable barrier functions are constructed showing that the leading term in the expansion for the solutions is the coordinate function $r$.

\begin{lemma}\label{lem:barrier}
There exist constants $C>0$ and $r_*>1$, such that for all $\rho> r_*$ we have
\begin{equation}\label{equ:C0}
|u_{\rho}-r|\leq C \quad\quad \text{ on }\quad\quad  M_{\rho}\setminus M_1.
\end{equation}
\end{lemma}

\begin{proof}
We start by constructing an upper barrier. Let $r_0>1$, $\lambda\in\mathbb{R}$, and $c_0>0$ be constants to be determined, and let $\rho>r_0$. Consider the spacetime harmonic function $w_{r_0}^+\in C^{2,\alpha}(\bar{M}_{r_0})$  with boundary conditions $w_{r_0}^+=1$ on $T_{r_0}\sqcup\left(\sqcup_{i=2}^m \partial^+_i M_{\rho} \right) \sqcup\left(\sqcup_{i=2}^\ell \partial^-_i M_{\rho}\right)$ and $w_{r_0}^+=0$ on $\partial^-_1 M_{\rho}$, that is guaranteed by the discussion in Section \ref{sec4}. By the Hopf lemma
\begin{equation}\label{equ:w+}
\mathbf{n}\left(w_{r_0}^+ \right)> 0\quad\quad \text{ on }\quad\quad  T_{r_0}\sqcup \left(\sqcup_{i=2}^m \partial^+_i M_{\rho} \right) \sqcup\left(\sqcup_{i=2}^\ell \partial^-_i M_{\rho} \right),
\end{equation}
where $\mathbf{n}$ denotes the unit outer normal. Define
\begin{equation}\label{definitionssang}
{z}^+\coloneqq \left\{ \begin{array}{cc}
c_0 w_{r_0}^+  & \text{on } M_{r_0},\\
r+(c_0-r_0-\lambda r_0^{-2})+ \lambda r^{-2} & \text{on } M\setminus M_{r_0}.
\end{array} \right.
\end{equation}
Clearly $z^+$ is continuous on $M$, and is $C^{2,\alpha}$ smooth away from $T_{r_0}$.

We now show that $z^+$ is a super solution on $M\setminus M_{r_0}$ if $\lambda$ and $r_0$ are chosen appropriately. Observe that \eqref{asymptotics} yields
\begin{equation}
\det g= r^2 \det \hat{g} \left(1+ (\mathrm{Tr}_{\hat{g}}\mathbf{m})r^{-3}+o(r^{-3})\right),\quad\quad\quad
g^{rr}=r^2 \left(1+o(r^{-3})\right),
\end{equation}
so that in the exterior region
\begin{equation}
\Delta z^+ =\frac{1}{\sqrt{\det g}}\partial_r \left( g^{rr} \sqrt{\det g} \partial_r z_+\right) +o(r^{-2})
=3r\left(1-\frac{1}{2}(\mathrm{Tr}_{\hat{g}}\mathbf{m})r^{-3}+o(r^{-3})\right).
\end{equation}
Moreover
\begin{equation}
\mathrm{Tr}_{g}k =-3 +(\mathrm{Tr}_{\hat{g}}\mathbf{p})r^{-3}+o(r^{-3}), \quad\quad\quad
|\nabla z^+|^2=g^{rr}\left(\partial_r z^+ \right)^2
=r^2\left(1-2\lambda r^{-3}+o(r^{-3})\right),
\end{equation}
so that
\begin{equation}
\left(\mathrm{Tr}_{g}k \right)|\nabla z^+|
=r\left(-3 +(6\lambda+\mathrm{Tr}_{\hat{g}}\mathbf{p})r^{-3}+o(r^{-3})\right).
\end{equation}
It follows that
\begin{equation}\label{akfnbsaohgj}
\Delta z^+ +\left(\mathrm{Tr}_{g}k \right)|\nabla z^+|=\left(6\lambda-\frac{3}{2}\mathrm{Tr}_{\hat{g}}\mathbf{m}
+\mathrm{Tr}_{\hat{g}}\mathbf{p}\right)r^{-2} +o(r^{-2})\leq 0
\end{equation}
on $M\setminus M_{r_0}$, if $\lambda$ is chosen so that the term in parentheses within \eqref{akfnbsaohgj} is less than -1 and $r_0$ is chosen sufficiently large.

The function $z^+$ is a super solution for the spacetime harmonic equation on $M_{r_0}$ and $M\setminus M_{r_0}$ separately. Moreover, it is a weak super solution on $M$ if $c_0$ is chosen appropriately. To see this, note that with the help of \eqref{equ:w+} we may choose $c_0>0$ large enough so that on $T_{r_0}$ the following inequality holds
\begin{equation}\label{joint}
c_0 \mathbf{n}\left(w_{r_0}^+\right)\geq \mathbf{n}\left(r + \lambda r^{-2}\right)=r_0^2\left(1-2\lambda r_0^{-3}+o(r_{0}^{-3})\right).
\end{equation}
Consider now the spacetime harmonic function $\tilde{u}_{\rho}\in C^{2,\alpha}(\overline{M}_{\rho})$ satisfying the boundary conditions
\begin{equation}
\left\{\begin{array}{cc}
\tilde{u}_\rho=\rho,\ &\text{on } \partial^+_1 M_\rho,\\
\tilde{u}_\rho=0,\ &\text{on } \partial^-_1 M_{\rho},\\
\tilde{u}_\rho=z^+,\ &\text{on } \left(\sqcup_{i=2}^m \partial^+_i M_{\rho} \right) \sqcup\left(\sqcup_{i=2}^\ell \partial^-_i M_{\rho} \right).
\end{array} \right.
\end{equation} 
Observe that in light of \eqref{akfnbsaohgj} and the definition \eqref{definitionssang}, the difference $z^+ -\tilde{u}_{\rho}$ is a super solution for a linear elliptic equation with bounded coefficients, namely
\begin{equation}\label{linearopelfk}
L(z^+ -\tilde{u}_{\rho}):=\Delta(z^+ -\tilde{u}_{\rho})
+\underbrace{\left(\mathrm{Tr}_g k\right)\frac{\nabla(z^+ +\tilde{u}_{\rho})}{|\nabla z^+|+|\nabla \tilde{u}_{\rho}|}}_{\vec{\mathcal{K}}}\cdot\nabla(z^+ -\tilde{u}_{\rho})\leq 0
\end{equation}
on $M_{r_0}$ and $M_{\rho}\setminus M_{r_0}$ separately. It follows that for nonnegative test functions $\varphi\in C^{\infty}_{c}(M_\rho)$ we have
\begin{align}
\begin{split}
0\leq &-\int_{M_{r_0}}\varphi L(z^+-\tilde{u}_{\rho})dV\\
=&\int_{M_{r_0}}\left(\nabla\varphi\cdot\nabla(z^+ -\tilde{u}_{\rho})
-\varphi\vec{\mathcal{K}}\cdot\nabla(z^+ -\tilde{u}_{\rho})\right)dV
-\int_{T_{r_0}}\varphi\mathbf{n}\left(c_0 w_{r_0}^+ -\tilde{u}_{\rho}\right)dA,
\end{split}
\end{align}
and
\begin{align}
\begin{split}
0\leq&-\int_{M_{\rho}\setminus M_{r_0}}\varphi L(z^+-\tilde{u}_{\rho})dV\\
=&\int_{M_{\rho}\setminus M_{r_0}}\left(\nabla\varphi\cdot\nabla(z^+ -\tilde{u}_{\rho})
-\varphi\vec{\mathcal{K}}\cdot\nabla(z^+ -\tilde{u}_{\rho})\right)dV
+\int_{T_{r_0}}\varphi\mathbf{n}\left(r+\lambda r^{-2} -\tilde{u}_{\rho}\right)dA,
\end{split}
\end{align}
so that summing these two inequalities produces
\begin{equation}
\int_{M_\rho}\left(\nabla\varphi\cdot\nabla(z^+ -\tilde{u}_{\rho})
-\varphi\vec{\mathcal{K}}\cdot\nabla(z^+ -\tilde{u}_{\rho})\right)dV
\geq\int_{T_{r_0}}\varphi \mathbf{n}\left(c_0 w_{r_0}^+ -r-\lambda r^{-2}\right)dA\geq 0.
\end{equation}
Thus, the weak maximum principle \cite[Theorem 8.1]{GT} implies that
\begin{equation}
\inf_{M_\rho}\left(z^+ -\tilde{u}_{\rho}\right)\geq
\inf_{\partial M_\rho}\left(z^+ -\tilde{u}_{\rho}\right)\geq 0,
\end{equation}
where in the last inequality we may ensure that $z^+\geq \tilde{u}_{\rho}$ on $\partial_1^+ M_{\rho}$ by choosing $c_0$ larger (dependent only on $r_0$ and $\lambda$) if necessary.

These estimates may be translated into bounds for $u_{\rho}$ in the following way. Since $z^+\geq \tilde{u}_{\rho}$ on $M_{\rho}$ we find that
$\mathbf{n}\left(\tilde{u}_\rho \right)\geq \mathbf{n}\left(z^+ \right)>0$ on $\left(\sqcup_{i=2}^m \partial^+_i M_{\rho} \right) \sqcup\left(\sqcup_{i=2}^\ell \partial^-_i M_{\rho} \right)$, where \eqref{equ:w+} was also used. Proposition \ref{lem:specialbd-app} may now be applied to $\rho^{-1}\tilde{u}_{\rho}$ and $\rho^{-1}u_{\rho}$ to find that $\tilde{u}_{\rho}\geq u_{\rho}$ on $M_{\rho}$. Hence $z^+\geq u_{\rho}$ on $M_{\rho}$ for $\rho> r_0$.

The construction of a lower barrier is analogous, so we will only give an outline. Let $r_1>1$, $\varsigma\in\mathbb{R}$, and $c_1>0$ be constants to be determined. Consider the spacetime harmonic function $w_{r_1}^-\in C^{2,\alpha}(\overline{M}_{r_1})$ with boundary conditions $w_{r_1}^-=-1$ on $T_{r_1}$ and $w_{r_1}^-=0$ on $\left(\sqcup_{i=2}^m \partial^+_i M_{\rho} \right) \sqcup\left(\sqcup_{i=1}^\ell \partial^-_i M_{\rho}\right)$. Use this function to define
\begin{equation}
{z}^-\coloneqq \left\{ \begin{array}{cc}
c_1 w_{r_1}^-  & \text{on } M_{r_1},\\
r-(c_1+r_1+\varsigma r_1^{-2})+ \varsigma r^{-2} & \text{on } M\setminus M_{r_1}.
\end{array} \right.
\end{equation}
The calculation \eqref{akfnbsaohgj} shows that if $\varsigma$ is chosen to ensure $6\varsigma-\frac{3}{2}\mathrm{Tr}_{\hat{g}}\mathbf{m}
+\mathrm{Tr}_{\hat{g}}\mathbf{p}>1$, then $r_1$ may be chosen sufficiently large to guarantee that $z^-$ is a sub solution of the spacetime harmonic equation on $M\setminus M_{r_1}$. The Hopf lemma, together with an appropriately large choice for $c_1$, shows that $z^-$ is a weak sub solution on $M$. A comparison argument then yields $z^- \leq u_{\rho}$ on $M_{\rho}$ for $\rho> r_1$. By setting $r_* =\max\{r_0 ,r_1\}$, we then have $z^-\leq u_{\rho}\leq z^+$ on $M_{\rho}$ for all $\rho> r_*$. The desired conclusion now follows.
\end{proof}

The pointwise estimates for $u_\rho$ may be parlayed into uniform gradient bounds in the asymptotic end with standard $L^p$-elliptic estimates.
These bounds, presented in the next result, will be sufficient to show convergence of the boundary integral on the right-hand side of \eqref{aifjq0gn2856635r} to the total energy.

\begin{lemma}\label{lem:grad}
Let $r_*>1$ be as in Lemma \ref{lem:barrier}. There exists a constant $C>0$ such that for all $\rho> r_*$ we have
\begin{equation}\label{equ:C0}
|\nabla u_{\rho}-\nabla r|\leq C \quad\quad \text{ on }\quad\quad  M_{\rho}\setminus M_1.
\end{equation}
\end{lemma}

\begin{proof}
Let $\rho> r_*$, and set $v_\rho=u_\rho-r$. As in \eqref{linearopelfk} we find that $v_\rho$ satisfies the equation
\begin{equation}
\Delta v_\rho+\left(\text{Tr}_g k\right)\frac{\nabla(u_{\rho}+r)}{|\nabla u_\rho|+|\nabla r|}\cdot\nabla v_\rho=-\Delta r-\left(\text{Tr}_g k\right)|\nabla r|=:F.
\end{equation}
Note that the first order coefficients are uniformly bounded. Take a point
$x_0\in \partial^+_1 M_\rho$, and let $B_\epsilon$ and $B_{\epsilon/2}$ be geodesic balls centered at $x_0$ with radii $\epsilon$ and $\epsilon/2$, respectively. We fix $\epsilon>0$ so that it is less than the injectivity radius for any point $x\in M\setminus M_{r_*}$. For $1<p<\infty$ the boundary $L^p$-estimates \cite[Theorem 9.13]{GT}, with $v_\rho =0$ on $\partial_1^+ M_{\rho}$, yield
\begin{equation}
\| v_\rho\|_{W^{2,p}(B_{\epsilon/2}\cap M_\rho )}\leq C_0 \left(\|F\|_{L^p(B_\epsilon\cap M_\rho)}+\|v_\rho \|_{L^p(B_\epsilon\cap M_\rho )}\right).
\end{equation}
Since $g$ is asymptotically locally hyperbolic, the constant $C_0$ is uniform over all $x_0\in \partial_1^+ M_{\rho}$ and all $\rho> r_*$. Furthermore, a calculation similar to \eqref{akfnbsaohgj} shows that
\begin{equation}
\Delta r +\left(\mathrm{Tr}_{g}k \right)|\nabla r|=\left(\mathrm{Tr}_{\hat{g}}\mathbf{p}
-\frac{3}{2}\mathrm{Tr}_{\hat{g}}\mathbf{m}
\right)r^{-2} +o(r^{-2}),
\end{equation}
and therefore $F$ is uniformly bounded in the asymptotic end. Moreover, by Lemma~\ref{lem:barrier} we have that $v_{\rho}$ is uniformly bounded on $M_{\rho}\setminus M_1$ independent of $\rho$. Thus, with the aid of Sobolev embedding (and choosing $p>3$) there is a uniform constant $C$ such that
\begin{equation}
\|v_\rho \|_{C^{1,\beta}(B_{\epsilon/2}\cap M_\rho)}\leq C_1\| v_\rho \|_{W^{2,p}(B_{\epsilon/2}\cap M_\rho)}\leq C,
\end{equation}
where $\beta=1-\tfrac{3}{p}$. The interior $L^p$-estimates may be used to obtain the same conclusion in balls away from the boundary. The desired result now follows.
\end{proof}

\begin{proof}[Proof of Theorem \ref{thm:PMT}]
Consider the inequality \eqref{aifjq0gn2856635r}. According to the hypotheses of the theorem, the left-hand side is nonnegative. Moreover, the assumptions imply that $M$ satisfies the homotopy condition with respect to some coordinate torus $T_{\bar{r}}$ in the asymptotic end. Since the entire asymptotic end is homotopy equivalent to $T_{\bar{r}}$, we find that $M_r$ satisfies the homotopy condition with respect to $\partial_1^+ M_r$. Therefore, the arguments of Section \ref{sec5} apply to show that the level set Euler characteristics satisfy $\chi(\Sigma_t^r)\leq 0$ for all regular values $t\in[0,r]$. To establish that the total energy $E$ is nonnegative, it then suffices to show that the boundary integral on the right-hand side of \eqref{aifjq0gn2856635r} converges to a positive multiple of this quantity, namely
\begin{equation}\label{fjkah29tgj}
E=\lim_{r\to\infty}-\frac{2}{|T^2|}\int_{\partial^+_1 M_r} \theta_{+}|\nabla u_r| dA.
\end{equation}
Observe that the asymptotics \eqref{asymptotics} imply that the unit outer normal to $\partial_1^+ M_r$ satisfies
\begin{equation}
\mathbf{n}=\frac{\nabla r}{|\nabla r|}=\frac{g^{rr}\partial_r +g^{rl}\partial_l}{|\nabla r|}=r\left(1+o(r^{-3})\right)\partial_r +\Sigma_{l=1}^{2}o(r^{-4})\partial_l,
\end{equation}
where $l=1,2$ denote directions tangential to the torus cross-sections. It follows that the mean curvature with respect to $\mathbf{n}$ has the expansion
\begin{align}\label{Hk1}
\begin{split}
H=&\frac{1}{2}\left(r^{-2} \hat{g}^{jl}-r^{-5}\textbf{m}^{jl}+o(r^{-5})\right)\mathbf{n}\left(r^2 \hat{g}_{jl}+r^{-1}\textbf{m}_{jl}+o(r^{-1})\right)\\
=&2-\frac{3}{2}\left(\text{Tr}_{\hat{g}}\textbf{m}\right)r^{-3}+o(r^{-3}),
\end{split}
\end{align}
where $\mathbf{m}^{jl}$ indicates indices raised with the metric $\hat{g}$.
Similarly we have
\begin{align}\label{Hk2}
\begin{split}
\text{Tr}_{\partial_1^+ M_r}k =&\left(r^{-2} \hat{g}^{jl}-r^{-5}\textbf{m}^{jl}+o(r^{-5})\right)\left(-r^2 \hat{g}_{jl}+r^{-1}\textbf{p}_{jl}-r^{-1}\textbf{m}_{jl}\right)\\
=&-2+\left(\text{Tr}_{\hat{g}}\textbf{p}\right)r^{-3}+o(r^{-3}),
\end{split}
\end{align}
so that
\begin{equation}\label{akfkakdllg}
\theta_+ =H+\text{Tr}_{\partial_1^+ M_r}k
=\left(\text{Tr}_{\hat{g}}\textbf{p}-\frac{3}{2}\text{Tr}_{\hat{g}}\textbf{m}\right)r^{-3}
+o(r^{-3}).
\end{equation}
Furthermore the area form is a direct calculation, and Lemma~\ref{lem:grad} yields the asymptotics for the modulus of the derivatives of the spacetime harmonic function
\begin{equation}\label{Hk3}
dA=r^2 \left(1+o(r^{-3})\right)dA_{\hat{g}},\quad\quad\quad\quad
|\nabla u_{r}|=|\nabla r|+O(1)=r\left(1+o(1)\right).
\end{equation}
Therefore, combining \eqref{akfkakdllg} and \eqref{Hk3} produces \eqref{fjkah29tgj}.
\end{proof}

\begin{proof}[Proof of Theorem~\ref{thm:rigidity}]
We will first establish an energy lower bound by taking a limit of inequality \eqref{aifjq0gn2856635r}. To accomplish this, observe that
by Lemma~\ref{lem:barrier}, the functions $u_r$ are locally uniformly bounded. Standard elliptic estimates then show that derivatives of $u_r$ are locally uniformly bounded in $C^{2,\alpha}$, for any $\alpha\in (0,1)$. Therefore a diagonal argument implies that there is an increasing sequence $r_{\mathrm{j}}\rightarrow\infty$, such that $\{u_{r_{\mathrm{j}}}\}$ converges in $C^{2,\alpha}$ on compact subsets to a spacetime harmonic function $u\in C^{2,\alpha}(M)$.

Consider now the limit of \eqref{aifjq0gn2856635r} as $r_{\mathrm{j}}\rightarrow\infty$. Due to the regular convergence of the spacetime harmonic functions, it is clear that the limit may be passed under the boundary integrals on the left-hand side, and the same holds for the bulk integrals over fixed compact subsets, except possibly the first involving the Hessian. To deal with the Hessian term, fix a compact set $\Omega\subset M$. For any $\varepsilon>0$, define $\Omega_{\varepsilon}=\{x\in \Omega\mid |\nabla u (x)|\geq \varepsilon\}$. Because $u_{r_{\mathrm{j}}}$ converges to $u$ in $C^{2,\alpha}(\Omega)$, for $\mathrm{j}$ large enough it holds that $|\nabla u_{r_{\mathrm{j}}}|\geq 2^{-1}\varepsilon$ on $\Omega_{\varepsilon}$, and therefore
\begin{equation}
\lim_{\mathrm{j}\rightarrow\infty}\frac{|\bar{\nabla}^2 u_{r_{\mathrm{j}}}|^2}{|\nabla u_{r_{\mathrm{j}}}|}(x)=\frac{|\bar{\nabla}^2 u|^2}{|\nabla u|}(x)\quad\quad\text{ for all } x\in \Omega_{\varepsilon}.
\end{equation}
Fatou's lemma then applies to yield
\begin{equation}
\liminf_{\mathrm{j}\to\infty} \int_{\Omega} \frac{|\bar{\nabla}^2 u_{r_{\mathrm{j}}}|^2}{|\nabla u_{r_{\mathrm{j}}}|}\, dV\geq  \liminf_{\mathrm{j}\to\infty} \int_{\Omega_{\varepsilon}} \frac{|\bar{\nabla}^2 u_{r_{\mathrm{j}}}|^2}{|\nabla u_{r_{\mathrm{j}}}|} dV\geq \int_{\Omega_{\varepsilon}} \frac{|\bar{\nabla}^2 u |^2}{|\nabla u|} dV,
\end{equation}
and by the monotone convergence theorem we may let $\varepsilon\rightarrow 0$
to obtain
\begin{equation}\label{alfkhgowpsg90}
\liminf_{\mathrm{j}\rightarrow\infty}\int_{\Omega}\frac{|\bar{\nabla}^2 u_{r_{\mathrm{j}}}|^2}{|\nabla u_{r_{\mathrm{j}}}|}dV\geq
\int_{\Omega}\frac{|\bar{\nabla}^2 u|^2}{|\nabla u|}dV.
\end{equation}
Therefore if $\Omega$ properly contains all boundary components of $M$, then utilizing the fact that all regular level sets satisfy the Euler characteristic estimate $\chi(\Sigma_t^{r_{\mathrm{j}}})\leq 0$, as well as \eqref{fjkah29tgj}, \eqref{alfkhgowpsg90} and the dominant energy condition, we find that taking the $\liminf$ of both sides in inequality \eqref{aifjq0gn2856635r} produces
\begin{align}\label{aogjnq0hbnn}
\begin{split}
E\geq & \frac{1}{|T^2|}\int_{\Omega}\left(\frac{|\bar{\nabla}^2 u|^2}{|\nabla u|}+2(\mu+J(\nu))|\nabla u|\right)dV\\
&
-\frac{2}{|T^2|}\sum_{i=2}^{m}\int_{\partial_i^+ M}\theta_{-} |\nabla u|dA
-\frac{2}{|T^2|}\sum_{i=1}^{\ell}\int_{\partial_i^- M}\theta_{+} |\nabla u|dA,
\end{split}
\end{align}
where $\nu=\nabla u/|\nabla u|$.

We are now in a position to establish the rigidity statement. The line of argument from this point is almost identical to the proof of Theorem \ref{thm:EGM}, and thus we will only give an outline here while emphasizing the differences. In particular when $E=0$, the lower bound \eqref{aogjnq0hbnn} for arbitrary $\Omega$ together with the strategy in Section \ref{sec5} establishes the following: $|\nabla u|>0$ on all of $M$, there is only one boundary component $\partial_1^- M$, $M$ is diffeomorphic to $[0,\infty)\times \Sigma$ for some orientable closed surface $\Sigma$, $\mu=|J|_g =-J(\nu)$, each level set $\Sigma_t=\{t\}\times\Sigma$ has vanishing null second fundamental form $\chi^+ =0$, and the metric may be expressed as
\begin{equation}
g=|\nabla u|^{-2}dt^2 +g_t.
\end{equation}
Next observe that \eqref{fjaijsroa}, \eqref{=-09}, and the Gauss-Bonnet theorem yield
\begin{equation}\label{19gns0-2hg}
0=2\pi \chi(\Sigma_t)-\int_{\Sigma_t}|\nabla_t \log f +X|^2 dA.
\end{equation}
Since $M$ satisfies the homotopy condition with respect to conformal infinity, Proposition \ref{lem:separable} may be applied to show that $\partial_1^- M$ cannot be separated from the asymptotic end by a 2-sphere. This implies that $\Sigma_t$ cannot be a 2-sphere. It follows that $\chi(\Sigma_t)\leq 0$ for all $t\geq 0$, and in fact by \eqref{19gns0-2hg} we must have $\chi(\Sigma_t)=0$ in addition to $X=-\nabla_t \log f$. From \eqref{fjaijsroa} we then find that $(\Sigma_t,g_t)$ is a flat torus for all $t\geq 0$. Finally, if $k=-g$ then the arguments in the last paragraph of Section \ref{sec5} show that
\begin{equation}
g=dt^2 +e^{2t} \hat{g}
\end{equation}
for some flat metric on $T^2$. Thus by changing radial coordinates $r=e^t$, we find that $(M,g)$ is isometric to the Kottler time slice $([1,\infty)\times T^2,b)$.
\end{proof}

\begin{proof}[Proof of Theorem~\ref{energylb}]
This result arises from an updated version of inequality \eqref{aogjnq0hbnn}.
Since the dominant energy condition is not assumed for the first portion of this theorem, \eqref{aogjnq0hbnn} should be amended in the terms involving $\mu$ and $J$. To accomplish this, observe that
\begin{equation}
(\mu+J(\nu))|\nabla u|\leq \left(|\nabla r|+O(1)\right)(\mu+|J|_g)\leq c r\left(\mu+|J|_g \right)\in L^1(M\setminus M_1)
\end{equation}
for some uniform constant $c$, where we have used Lemma \ref{lem:grad}.
Then by the dominated convergence theorem
\begin{equation}
\liminf_{\mathrm{j}\rightarrow\infty}\int_{M_{r_{\mathrm{j}}}}\left(\mu|\nabla u_{r_{\mathrm{j}}}|+J(\nabla u_{r_{\mathrm{j}}})\right)dV
=\int_{M}\left(\mu+J(\nu)\right)|\nabla u|dV.
\end{equation}
Thus, the portion of \eqref{aogjnq0hbnn} involving the energy/momentum densities should in this case be replaced by an integration over all of $M$.
Now for the Hessian term, we may take a sequence of $\Omega$ that exhaust $M$ to find
\begin{align}\label{aogjnq0hbnn1}
\begin{split}
E\geq & \frac{1}{|T^2|}\int_{M}\left(\frac{|\bar{\nabla}^2 u|^2}{|\nabla u|}+2(\mu+J(\nu))|\nabla u|\right)dV\\
&
-\frac{2}{|T^2|}\sum_{i=2}^{m}\int_{\partial_i^+ M}\theta_{-} |\nabla u|dA
-\frac{2}{|T^2|}\sum_{i=1}^{\ell}\int_{\partial_i^- M}\theta_{+} |\nabla u|dA.
\end{split}
\end{align}
The hypotheses of the theorem guarantee that the boundary integrals are nonnegative, and this yields the desired inequality \eqref{thm1.1}.

Consider now the case in which $k=-g$, the dominant energy condition holds, and the boundary is minimal $H=0$ instead of weakly trapped. In this situation, we relabel the boundary components of $M$ so that all are within the $\partial_i^- M$ designation. This changes the spacetime harmonic function boundary conditions according to \eqref{equ:boundarycondition11}, \eqref{equ:boundarycondition21}, and produces a version of \eqref{aogjnq0hbnn1} without terms involving $\partial_i^+ M$. Since $\theta_+ =H+\mathrm{Tr}_{\partial_i^- M}k=-2$ for $i\in \llbracket 1,\ell+m-1 \rrbracket$, and the Hopf lemma ensures that $|\nabla u|=\partial_{\upsilon} u>0$ on $\partial_1^- M$, it follows that
\begin{equation}
E\geq -\frac{2}{|T^2|}\sum_{i=1}^{\ell+m-1} \int_{\partial^-_i M }\theta_+|\nabla u|dA\geq \mathcal{C}\frac{|\partial^-_1 M|}{|T^2|}
\end{equation}
where $\mathcal{C}=4\min_{\partial_1^- M}\partial_{\upsilon}u>0$.
\end{proof}

\section{An Example}\label{sec:example}
\label{sec7}
\setcounter{equation}{0}
\setcounter{section}{7}

In this section we illustrate two of the main theorems with explicit initial data, and in the process show the necessity of certain hypotheses. More precisely, we construct initial data $(M,g,k)$ satisfying the hypotheses of Theorem \ref{thm:rigidity} or \ref{thm:EGM} minus the assumption on the structure of $k$, while additionally exhibiting a vanishing mass aspect function (in the noncompact case) and vanishing energy and momentum densities $\mu=|J|_g=0$. It is then shown that, unlike the conclusion of Theorems \ref{thm:rigidity} and \ref{thm:EGM}, the metric $g$ does not have a warped product structure, and in a departure from the conclusion of \cite[Theorem 6.1]{EGM} the initial data arise from a vacuum (with zero cosmological constant) pp-wave spacetime which is not flat.

Fix $r_0>1$ and $P_\theta, P_{\xi}>0$, and consider the 4-manifold $N=\mathbb{R} \times [r_0,\infty) \times T^2$ equipped with the Lorentzian metric
\begin{equation}
\tilde{g}=-2(1-r^{-3})^{-1/2}d\tau dr+r^{-2}(1-r^{-3})^{-1}dr^2+r^2(1-r^{-3})d\xi^2+r^2d\theta^2.
\end{equation}
Here $\tau$ and $r$ are coordinates on $\mathbb{R}$ and $[r_0,\infty)$ respectively, and $\xi$ and $\theta$ are coordinates on $T^2$ with periods $P_\xi$ and $P_\theta$ respectively. Note that $N$ has a boundary $\{r=r_0\}$. A calculation shows that $(N,\tilde{g})$ is a vacuum but non-flat spacetime (with zero cosmological constant). Moreover, consider a function $u=u(r)$ with $u(r_0)=0$ defined by
\begin{equation}
\frac{d u}{dr}=(1-r^{-3})^{-1/2}.
\end{equation}
It follows that the spacetime gradient $\tilde{\nabla}u=-\partial_{\tau}$ is a null Killing field, so that
\begin{equation}\label{equ:zerohessian}
\tilde{g}(\tilde{\nabla}u,\tilde{\nabla}u)=0,\quad\quad\quad\quad
\tilde{\nabla}^2 u=0.
\end{equation}
In particular, $(N,\tilde{g})$ is a pp-wave spacetime.
		
Let $M=\{\tau=0\}\subset N$, then the induced metric on $M$ is given by
\begin{equation}
g=r^{-2}(1-r^{-3})^{-1}dr^2+r^2(1-r^{-3})d\xi^2+r^2d\theta^2.
\end{equation}
Observe that if $P_\xi=4\pi/3$, the metric $g$ can be extended smoothly to $r=1$, and in this case $(M,g)$ is called the \textit{AdS soliton/Horowitz-Myers geon} \cite{HM,Ewoolgar}. Because we stay away from $r=1$, we do not include a restriction on $P_\xi$. The unit timelike normal to $M$ is $\tilde{\mathbf{n}}=-r^{-1}\partial_{\tau}-r(1-r^{-3})^{1/2}\partial_r$, which yields the second fundamental form
\begin{equation}
k=\frac{1}{2}\mathcal{L}_{\tilde{\mathbf{n}}}\tilde{g}
=-r^{-2}(1-r^{-3})^{-1/2}dr^2-r^2(1-r^{-3})^{1/2}(1+2^{-1}r^{-3})d\xi^2
-r^2(1-r^{-3})^{1/2}d\theta^2
\end{equation}
where $\mathcal{L}$ denotes Lie differentiation. Since $(N,\tilde{g})$ is vacuum, we have $\mu=|J|_g=0$. Furthermore, equations \eqref{equ:zerohessian} imply that the function $u$, when restricted to $M$, satisfies the vanishing spacetime Hessian property
\begin{equation}
\bar{\nabla}^2 u=\nabla^2 u+|\nabla u|k=0.
\end{equation}
As in \eqref{gajaiogj}, this shows that each level set $\Sigma_t=\{u=t,\tau=0\}$ has vanishing null expansion $\chi^+=0$, and therefore $M$ is foliated by MOTS. In addition, the functions $c_r u$ coincide with the spacetime harmonic functions satisfying the boundary value problem \eqref{equ:boundarycondition11}, for some constants $c_r\rightarrow 1$ as $r\rightarrow\infty$.
		
Lastly, we show that the mass aspect function of $(M,g,k)$ vanishes. In order to accomplish this we make a change of radial coordinate, as in  \cite[(2.23)]{Ewoolgar}, in order to place the metric into a form satisfying the asymptotics \eqref{asymptotics}. Namely, define
\begin{equation}
\rho\coloneqq 4^{-1/3} r^{-1} \left[1-(1-r^{-3})^{1/2} \right]^{-2/3}.
\end{equation}
Then in this coordinate the metric has the expansion
\begin{align}
\begin{split}
g=&\rho^{-2}d\rho^2+\rho^{2} (1+4^{-1}\rho^{-3})^{4/3}\left( \frac{1-4^{-1}\rho^{-3}}{1+4^{-1}\rho^{-3}} \right)^2 d\xi^2+ \rho^2(1+4^{-1}\rho^{-3})^{4/3} d\theta^2\\
=&\rho^{-2}d\rho^2+\rho^2\hat{g}+\rho^{-1}\mathbf{m}+Q_g,
\end{split}
\end{align}
where $Q_g$ satisfies \eqref{alk3nhapgaj} and
\begin{equation}
\hat{g}=d\xi^2+d\theta^2,\quad\quad\quad\quad \textbf{m}=-\frac{2}{3}d\xi^2+\frac{1}{3}d\theta^2.
\end{equation}
In particular, $\mathrm{Tr}_{\hat{g}}\textbf{m}=-\frac{1}{3}$. Furthermore,
the second fundamental form has the expansion
\begin{align}
\begin{split}
k= &-\rho^{-2} d\rho^2-\rho^{2}\hat{g}+\left[ 2^{-1}\rho^{-5}d\rho^2 -\rho^{-1}\left( \frac{1}{3}d\xi^2-\frac{1}{6}d\theta^2 \right)\right]+\tilde{Q}_k\\
=&-g+\rho^{-1}\mathbf{p}
+2^{-1}\rho^{-5}d\rho^2 +Q_k
\end{split}
\end{align}
where $Q_k$ satisfies \eqref{alk3nhapgaj} and
\begin{equation}
\mathbf{p}=-d\xi^2+\frac{1}{2}d\theta^2.
\end{equation}
In particular $\mathrm{Tr}_{\hat{g}}\textbf{p}=-\frac{1}{2}$, and therefore the mass aspect function $\mathrm{Tr}_{\hat{g}}\left(3\mathbf{m}-2\mathbf{p}\right)=0$. Note that due to the presence of the term $\rho^{-5}d\rho^2$, the extrinsic curvature does not satisfy \eqref{asymptotics}. However, all the results of this manuscript continue to hold under the slightly weaker asymptotics in which the radial direction decay is amended to $k_{\rho\rho}+g_{\rho\rho}=O(\rho^{-5})$, with corresponding fall-off for derivatives.

\end{document}